\documentclass{article}
\usepackage{mathtools,amssymb,amsthm,tikz,tikz-cd}
\usepackage{fullpage}
\usepackage{etoolbox}
\usepackage{bbm}
\usepackage{extarrows}
\usepackage{enumitem}
\usepackage{makecell}
\usepackage{tensor}
\numberwithin{equation}{section}
\usepackage[english]{babel}
\usepackage[utf8]{inputenc}
\usepackage{amsmath}
\usepackage{bm}
\usepackage{graphicx}
\newtheorem{theorem}{Theorem}[section]

\newtheorem{proposition}[theorem]{Proposition}
\theoremstyle{remark}

\theoremstyle{definition}

\title{Colored stochastic vertex models with U-turn boundary}

\author{Chenyang Zhong \thanks{Department of Statistics, Columbia University}}

\date{\today}

\begin{document}
\maketitle

\newcommand{\caps}[2]{\begin{tikzpicture}
  \draw[->] (0,-0.5) arc(-90:90:0.5);
  \filldraw[black] (0.5,0) circle (2pt);
  \node at (0,-0.5) [anchor=east] {$#1$};
  \node at (0.5,0) [anchor=west] {$x$};
  \node at (0,0.5) [anchor=east] {$#2$};
\end{tikzpicture}}

\newcommand{\capsN}[2]{\begin{tikzpicture}
  \draw[->] (0,-0.5) arc(-90:90:0.5);
  \filldraw[black] (0.5,0) circle (2pt);
  \node at (0,-0.5) [anchor=east] {$#1$};
  \node at (0.5,0) [anchor=west] {$x_n$};
  \node at (0,0.5) [anchor=east] {$#2$};
\end{tikzpicture}}
\newcommand{\capsm}[2]{\begin{tikzpicture}
  \draw[->] (0,-0.5) arc(-90:90:0.5);
  \filldraw[black] (0.5,0) circle (2pt);
  \node at (0,-0.5) [anchor=east] {$#1$};
  \node at (0.5,0) [anchor=west] {$x_n^{-1}$};
  \node at (0,0.5) [anchor=east] {$#2$};
\end{tikzpicture}}

\newcommand{\capsC}[3]{\begin{tikzpicture}
   \draw (0,-0.5) to [out = 0, in = 180] (0.5,0);
  \draw[<-] (0,0.5) to [out = 0, in = 180] (0.5,0);
  \filldraw[black] (0.5,0) circle (2pt);
  \node at (0,-0.5) [anchor=east] {$#2$};
  \node at (0,0.5) [anchor=east] {$#1$};
  \node at (0.5,0)   [anchor=west] {$#3$};
\end{tikzpicture}}

\newcommand{\newcaps}[2]{\begin{tikzpicture}
  \draw (0,-0.5) to [out = 0, in = 180] (0.5,0);
  \draw[<-] (0,0.5) to [out = 0, in = 180] (0.5,0);
  \filldraw[black] (0.5,0) circle (2pt);
  \node at (0,-0.5) [anchor=east] {$#2$};
  \node at (0.5,0) [anchor=west] {$x_n$};
  \node at (0,0.5) [anchor=east] {$#1$};
\end{tikzpicture}}

\newcommand{\gammaicen}[4]{\begin{tikzpicture}
\coordinate (a) at (-.75, 0);
\coordinate (b) at (0, .75);
\coordinate (c) at (.75, 0);
\coordinate (d) at (0, -.75);
\coordinate (aa) at (-.75,.5);
\coordinate (cc) at (.75,.5);
\draw (a)--(c);
\draw (b)--(d);
\draw[fill=white] (a) circle (.25);
\draw[fill=white] (b) circle (.25);
\draw[fill=white] (c) circle (.25);
\draw[fill=white] (d) circle (.25);
\node at (0,1) { };
\node at (a) {$#2$};
\node at (b) {$#3$};
\node at (c) {$#4$};
\node at (d) {$#1$};
\path[fill=white] (0,0) circle (.2);
\node at (0,0) {$x_n$};
\end{tikzpicture}}

\newcommand{\gammaicenm}[4]{\begin{tikzpicture}
\coordinate (a) at (-.75, 0);
\coordinate (b) at (0, .75);
\coordinate (c) at (.75, 0);
\coordinate (d) at (0, -.75);
\coordinate (aa) at (-.75,.5);
\coordinate (cc) at (.75,.5);
\draw (a)--(c);
\draw (b)--(d);
\draw[fill=white] (a) circle (.25);
\draw[fill=white] (b) circle (.25);
\draw[fill=white] (c) circle (.25);
\draw[fill=white] (d) circle (.25);
\node at (0,1) { };
\node at (a) {$#2$};
\node at (b) {$#3$};
\node at (c) {$#4$};
\node at (d) {$#1$};
\path[fill=white] (0,0) circle (.2);
\node at (0,0) {$x_n$};
\end{tikzpicture}}

\newcommand{\gammaicenn}[4]{\begin{tikzpicture}
\coordinate (a) at (-.75, 0);
\coordinate (b) at (0, .75);
\coordinate (c) at (.75, 0);
\coordinate (d) at (0, -.75);
\coordinate (aa) at (-.75,.5);
\coordinate (cc) at (.75,.5);
\draw (a)--(c);
\draw (b)--(d);
\draw[fill=white] (a) circle (.25);
\draw[fill=white] (b) circle (.25);
\draw[fill=white] (c) circle (.25);
\draw[fill=white] (d) circle (.25);
\node at (0,1) { };
\node at (a) {$#2$};
\node at (b) {$#3$};
\node at (c) {$#4$};
\node at (d) {$#1$};
\path[fill=white] (0,0) circle (.2);
\node at (0,0) {$x_n^{-1}$};
\end{tikzpicture}}

\newcommand{\gammaicei}[4]{\begin{tikzpicture}
\coordinate (a) at (-.75, 0);
\coordinate (b) at (0, .75);
\coordinate (c) at (.75, 0);
\coordinate (d) at (0, -.75);
\coordinate (aa) at (-.75,.5);
\coordinate (cc) at (.75,.5);
\draw (a)--(c);
\draw (b)--(d);
\draw[fill=white] (a) circle (.25);
\draw[fill=white] (b) circle (.25);
\draw[fill=white] (c) circle (.25);
\draw[fill=white] (d) circle (.25);
\node at (0,1) { };
\node at (a) {$#1$};
\node at (b) {$#2$};
\node at (c) {$#3$};
\node at (d) {$#4$};
\end{tikzpicture}}

\newcommand{\gammaice}[4]{\begin{tikzpicture}
\coordinate (a) at (-.8, 0);
\coordinate (b) at (0, .8);
\coordinate (c) at (.8, 0);
\coordinate (d) at (0, -.8);
\draw[->] (-0.6,0)--(0.6,0);
\draw[->,line width=0.5mm] (0,-0.6)--(0,0.6);
% \draw (a) circle (.25);
% \draw[fill=white] (b) circle (.25);
% \draw[fill=white] (c) circle (.25);
% \draw[fill=white] (d) circle (.25);
\node at (0,1) { };
\node at (a) {$#2$};
\node at (b) {$#3$};
\node at (c) {$#4$};
\node at (d) {$#1$};
\path[fill=white] (0,0) circle (.2);
\node at (0,0) {$x$};
\end{tikzpicture}}

\newcommand{\deltaice}[4]{\begin{tikzpicture}
\coordinate (a) at (-.8, 0);
\coordinate (b) at (0, .8);
\coordinate (c) at (.8, 0);
\coordinate (d) at (0, -.8);
\draw[<-] (-0.6,0)--(0.6,0);
\draw[->,line width=0.5mm] (0,-0.6)--(0,0.6);
% \draw (a) circle (.25);
% \draw[fill=white] (b) circle (.25);
% \draw[fill=white] (c) circle (.25);
% \draw[fill=white] (d) circle (.25);
\node at (0,1) { };
\node at (a) {$#4$};
\node at (b) {$#3$};
\node at (c) {$#2$};
\node at (d) {$#1$};
\path[fill=white] (0,0) circle (.2);
\node at (0,0) {$x$};
\end{tikzpicture}}

\newcommand{\gammaa}{\begin{tikzpicture}
\coordinate (a) at (-.75, 0);
\coordinate (b) at (0, .75);
\coordinate (c) at (.75, 0);
\coordinate (d) at (0, -.75);
\coordinate (aa) at (-.75,.5);
\coordinate (cc) at (.75,.5);
\draw (a)--(c);
\draw (b)--(d);
\draw[fill=white] (a) circle (.25);
\draw[fill=white] (b) circle (.25);
\draw[fill=white] (c) circle (.25);
\draw[fill=white] (d) circle (.25);
\node at (0,1) { };
\node at (a) {$+$};
\node at (b) {$+$};
\node at (c) {$+$};
\node at (d) {$+$};
\path[fill=white] (0,0) circle (.2);
\node at (0,0) {$z_i$};
\end{tikzpicture}}

\newcommand{\gammaA}[8]{\begin{tikzpicture}
\coordinate (a) at (-.75, 0);
\coordinate (b) at (0, .75);
\coordinate (c) at (.75, 0);
\coordinate (d) at (0, -.75);
\coordinate (aa) at (-.75,.5);
\coordinate (cc) at (.75,.5);
\draw[line width=0.5mm, #1] (a)--(0,0);
\draw[line width=0.6mm, #2] (b)--(0,0);
\draw[line width=0.5mm, #3] (c)--(0,0);
\draw[line width=0.6mm, #4] (d)--(0,0);
\draw[line width=0.5mm, #1,fill=white] (a) circle (.25);
\draw[line width=0.5mm, #2,fill=white] (b) circle (.25);
\draw[line width=0.5mm, #3,fill=white] (c) circle (.25);
\draw[line width=0.5mm, #4,fill=white] (d) circle (.25);
\node at (0,1) { };
\node at (a) {$#5$};
\node at (b) {$#6$};
\node at (c) {$#7$};
\node at (d) {$#8$};
\path[fill=white] (0,0) circle (.2);
\node at (0,0) {$z_i$};
\end{tikzpicture}}

\newcommand{\gammaAA}[4]{\begin{tikzpicture}
\coordinate (a) at (-.75, 0);
\coordinate (b) at (0, .75);
\coordinate (c) at (.75, 0);
\coordinate (d) at (0, -.75);
\coordinate (aa) at (-.75,.5);
\coordinate (cc) at (.75,.5);
\draw (a)--(c);
\draw (b)--(d);
\draw[line width=0.5mm, #1,fill=white] (a) circle (.25);
\draw[line width=0.5mm, #2,fill=white] (b) circle (.25);
\draw[line width=0.5mm, #2,fill=white] (c) circle (.25);
\draw[line width=0.5mm, #1,fill=white] (d) circle (.25);
\node at (0,1) { };
\node at (a) {$#3$};
\node at (b) {$#4$};
\node at (c) {$#4$};
\node at (d) {$#3$};
\path[fill=white] (0,0) circle (.2);
\node at (0,0) {$z_i$};
\draw [line width=0.5mm,#1] (-.5,0) to [out=0,in=90] (0,-.5);
\draw [line width=0.5mm,#2] (0,.5) to [out=-90,in=180] (.5,0);
\end{tikzpicture}}

\newcommand{\gammab}[2]{\begin{tikzpicture}
\coordinate (a) at (-.75, 0);
\coordinate (b) at (0, .75);
\coordinate (c) at (.75, 0);
\coordinate (d) at (0, -.75);
\coordinate (aa) at (-.75,.5);
\coordinate (cc) at (.75,.5);
\draw (a)--(c);
\draw [line width=0.6mm, #1] (b)--(d);
\draw[fill=white] (a) circle (.25);
\draw[line width=0.5mm, #1, fill=white] (b) circle (.25);
\draw[fill=white] (c) circle (.25);
\draw[line width=0.5mm, #1, fill=white] (d) circle (.25);
\node at (0,1) { };
\node at (a) {$+$};
\node at (b) {$#2$};
\node at (c) {$+$};
\node at (d) {$#2$};
\path[fill=white] (0,0) circle (.2);
\node at (0,0) {$z_i$};
\end{tikzpicture}}

\newcommand{\gammaB}[2]{\begin{tikzpicture}
\coordinate (a) at (-.75, 0);
\coordinate (b) at (0, .75);
\coordinate (c) at (.75, 0);
\coordinate (d) at (0, -.75);
\coordinate (aa) at (-.75,.5);
\coordinate (cc) at (.75,.5);
\draw [line width=0.5mm,#1] (a)--(c);
\draw (b)--(d);
\draw[line width=0.5mm, #1, fill=white] (a) circle (.25);
\draw[fill=white] (b) circle (.25);
\draw[line width=0.5mm, #1, fill=white] (c) circle (.25);
\draw[fill=white] (d) circle (.25);
\node at (0,1) { };
\node at (a) {$#2$};
\node at (b) {$+$};
\node at (c) {$#2$};
\node at (d) {$+$};
\path[fill=white] (0,0) circle (.2);
\node at (0,0) {$z_i$};
\end{tikzpicture}}

\newcommand{\gammac}[2]{\begin{tikzpicture}
\coordinate (a) at (-.75, 0);
\coordinate (b) at (0, .75);
\coordinate (c) at (.75, 0);
\coordinate (d) at (0, -.75);
\coordinate (aa) at (-.75,.5);
\coordinate (cc) at (.75,.5);
\draw (a)--(c);
\draw (b)--(d);
\draw[line width=0.5mm, #1, fill=white] (a) circle (.25);
\draw[fill=white] (b) circle (.25);
\draw[fill=white] (c) circle (.25);
\draw[line width=0.5mm, #1, fill=white] (d) circle (.25);
\node at (0,1) { };
\node at (a) {$#2$};
\node at (b) {$+$};
\node at (c) {$+$};
\node at (d) {$#2$};
\path[fill=white] (0,0) circle (.2);
\node at (0,0) {$z_i$};
\draw [line width=0.5mm,#1] (-.5,0) to [out=0,in=90] (0,-.5);
\end{tikzpicture}}

\newcommand{\gammaC}[2]{\begin{tikzpicture}
\coordinate (a) at (-.75, 0);
\coordinate (b) at (0, .75);
\coordinate (c) at (.75, 0);
\coordinate (d) at (0, -.75);
\coordinate (aa) at (-.75,.5);
\coordinate (cc) at (.75,.5);
\draw (a)--(c);
\draw (b)--(d);
\draw[fill=white] (a) circle (.25);
\draw[line width=0.5mm, #1, fill=white] (b) circle (.25);
\draw[line width=0.5mm, #1, fill=white] (c) circle (.25);
\draw[fill=white] (d) circle (.25);
\node at (0,1) { };
\node at (a) {$+$};
\node at (b) {$#2$};
\node at (c) {$#2$};
\node at (d) {$+$};
\path[fill=white] (0,0) circle (.2);
\node at (0,0) {$z_i$};
\draw [line width=0.5mm,#1] (0,.5) to [out=-90,in=180] (.5,0);
\end{tikzpicture}}

\begin{abstract}
In this paper, we introduce a class of colored stochastic vertex models with U-turn right boundary. The vertex weights in the models satisfy the Yang-Baxter equations and the reflection equation. Based on these equations, we derive recursive relations for partition functions of the models.
\end{abstract}

\section{Introduction}\label{Sect.1}

Exactly solvable lattice models, since the investigations by Baxter (\cite{Bax2,Bax1}), have been applied to various fields of mathematics and physics. For example, these models have found applications in algebraic combinatorics (e.g. \cite{BBBG,BSW,Kup,Kup2,WZ}), integrable probability (e.g. \cite{MR3466163,BP,BW,CP,OP}), and quantum field theory (e.g. \cite{BPZ,DMS}). 

Stochastic vertex models form an important class of exactly solvable lattice models. This class of models involve stochastic vertex weights, and play an important role in analyzing large scale and long time asymptotic behavior of Markovian systems, prominent examples of which include the Asymmetric Simple Exclusion Process (ASEP) and the Kardar-Parisi-Zhang equation. Among such models, the stochastic six-vertex model is introduced in \cite{MR1147356}, and its asymptotic properties are analyzed in \cite{MR3466163}. As a generalization of the stochastic six-vertex model, the stochastic higher spin vertex model is introduced and analyzed in \cite{Bor,BP,MR4283759,CP,MR4126935,OP}. The stochastic higher spin vertex model unifies many integrable models in nonequilibrium statistical mechanics, including the ASEP and directed polymers in random media.

More recently, an additional attribute called ``color'' was introduced to stochastic vertex models \cite{BW,MR3584389,kuniba2016stochastic}. These colored models are related to the quantum group $U_q(\widehat{\mathfrak{sl}_{n+1}})$, and can be interpreted as ensembles of colored paths on a rectangular lattice. Under specialization of parameters, the models degenerate to multi-species interacting particle systems such as multi-species ASEP. An extensive algebraic analysis of these models are carried out in \cite{BW}, and integral representations for natural observables of the models are obtained in \cite{MR4408007,MR4299137}. We refer to e.g. \cite{MR4260463,MR4400934,MR4317703,he2022shift} for further developments, including shift invariance and extension to half space models. 

In this paper, we introduce a class of colored stochastic vertex models with U-turn right boundary. The models can be interpreted as a collection of paths carrying ``signed colors''. Under the models, colored paths alternately bend to the right and to the left on the rectangular lattice, and have nontrivial interactions with the U-turn boundary (after passing through the U-turn, the signed color that a path carries may stay the same or change to the opposite sign). The partition functions of the models satisfy explicit recursive relations, which can be further related to the Noumi representation for affine Hecke algebra of type $\tilde{C}_n$. 

The rest of this paper is organized as follows. In Section \ref{Sect.2}, we describe the colored stochastic vertex model with U-turn boundary. We also present the Yang-Baxter equations and the reflection equation that are satisfied by the vertex weights involved in the model in this section. In Section \ref{Sect.3}, we use the Yang-Baxter equations and the reflection equation to derive recursive relations for partition functions of the models.

\subsection{Acknowledgement}

The author wishes to thank Daniel Bump and Jimmy He for helpful conversations and comments.

\section{Colored stochastic vertex model with U-turn boundary}\label{Sect.2}

In this section, we introduce colored stochastic vertex model with U-turn boundary. The model is described in Section \ref{Sect.2.1}. Then in Section \ref{Sect.2.2}, we introduce three types of $R$ vertices. The vertices involved in the model and the $R$ vertices satisfy three sets of Yang-Baxter equations and the reflection equation, which we introduce in Section \ref{Sect.2.3}. Throughout this section, we fix $n,L\in \mathbb{N}^{*}$.

\subsection{The model}\label{Sect.2.1}

We introduce $2n$ ``signed colors'' $1,\cdots,n,\overline{n},\cdots,\overline{1}$, and denote by $[\pm n]$ the set of these colors. We denote by $[n]:=\{1,\cdots,n\}$ the set of positive colors and $[\overline{n}]:=\{\overline{1},\cdots,\overline{n}\}$ the set of negative colors. We also introduce the color $0$ and identify it with the $+$ spin in the sequel. The ordering of the colors is taken to be $1<\cdots<n<0<\overline{n}<\cdots<\overline{1}$. The hyperoctahedral group $B_n$ is the set of signed permutations, that is, permutations $\sigma$ of $[\pm n]$ such that $\sigma(\overline{i})=\overline{\sigma(i)}$ for any $i \in[n]$ (where $\overline{\overline{i}}:=i$ for any $i\in [n]$).

In the following, we introduce the colored stochastic vertex model with U-turn boundary. The model depends on the parameters $q,\nu,t\in \mathbb{C}$, $s\in\mathbb{C}^{*}$, and $\bm{x}=(x_1,\cdots,x_n)\in \mathbb{C}^n$, where $s$ is the spin parameter and $x_1,\cdots,x_n$ are the spectral parameters. 

The model is based on a rectangular lattice with $2n$ rows and $L$ columns, where the rows are labeled $1,2,\cdots,2n$ from top to bottom, and the columns are labeled $1,2,\cdots,L$ from right to left. Each horizontal edge in the lattice is assigned a $+$ spin (which we identify with the color $0$) or a color from $[\pm n]$. Each vertical edge in the lattice is associated with a vector $\bm{I}=(I_1,\cdots,I_n,I_{\overline{n}},\cdots,I_{\overline{1}})$, where $I_i\in \mathbb{N}$ for any $i\in [\pm n]$. 

In the rest of this paper, we denote the standard basis of $\mathbb{C}^{2n}$ by $\bm{e}_1,\cdots,\bm{e}_n,\bm{e}_{\overline{n}},\cdots,\bm{e}_{\overline{1}}$. We also define $\bm{e}_0:=(0,\cdots,0)\in \mathbb{C}^{2n}$. For any $\bm{I}=(I_1,\cdots,I_n,I_{\overline{n}},\cdots,I_{\overline{1}})\in\mathbb{N}^{2n}$ and any $j,j'\in \{0\}\cup [\pm n]$, we define $\bm{I}_{[j,j']}:=\sum\limits_{k\in [\pm n]: j\leq k\leq j'} I_{k}$, and define $\bm{I}_{<j},\bm{I}_{> j},\bm{I}_{\leq j},\bm{I}_{\geq j}$ similarly.

To each vertex in the rectangular lattice, we associate a Boltzmann weight, which depends on the type of the vertex ($\Gamma$ or $\Delta$ vertex, see below) and the colors and vectors assigned to the horizontal and vertical edges that are adjacent to the vertex. A $\Gamma$ vertex with spectral parameter $x\in\mathbb{C}$ and adjacent edges given as in Figure \ref{figure2.1} is labeled by $(\bm{I},j,\bm{K},\ell)$ (where $j,\ell\in \{0\}\cup [\pm n]$ and $\bm{I}=(I_1,\cdots,I_n,I_{\overline{n}},\cdots,I_{\overline{1}}),\bm{K}=(K_1,\cdots,K_n,K_{\overline{n}},\cdots,K_{\overline{1}})\in \mathbb{N}^{2n}$), and its Boltzmann weight $\mathcal{L}_x(\bm{I},j,\bm{K},\ell)$ is given as in Table \ref{table2.1} if $\bm{I}+\bm{e}_j=\bm{K}+\bm{e}_{\ell}$ (otherwise $\mathcal{L}_x(\bm{I},j,\bm{K},\ell)=0$). A $\Delta$ vertex with spectral parameter $x\in\mathbb{C}$ and adjacent edges given as in Figure \ref{figure2.2} is labeled by $(\bm{I},j,\bm{K},\ell)$ (where $j,\ell\in \{0\}\cup [\pm n]$ and $\bm{I}=(I_1,\cdots,I_n,I_{\overline{n}},\cdots,I_{\overline{1}}),\bm{K}=(K_1,\cdots,K_n,K_{\overline{n}},\cdots,K_{\overline{1}})\in \mathbb{N}^{2n}$), and its Boltzmann weight $\mathcal{M}_x(\bm{I},j,\bm{K},\ell)$ is given as in Table \ref{table2.2} if $\bm{I}+\bm{e}_j=\bm{K}+\bm{e}_{\ell}$ (otherwise $\mathcal{M}_x(\bm{I},j,\bm{K},\ell)=0$). Each even numbered row in the lattice is a row of $\Gamma$ vertices, and each odd numbered row is a row of $\Delta$ vertices. For each $i\in [n]$, the spectral parameter for vertices in the $(2i-1)$th row and the $2i$th row is given by $x_i$. For each $i\in [n]$, the $(2i-1)$th row and the $2i$th row is connected by a ``cap vertex'' with spectral parameter $x_i$ on the right. The Boltzmann weights for a cap vertex with spectral parameter $x\in \mathbb{C}$ is given in Figure \ref{figure2.3}, where 
\begin{equation*}
    \phi(x):=\frac{1-x^2}{(1-\nu t x)(1+\nu^{-1} x)}.
\end{equation*}

\begin{figure}[!ht]
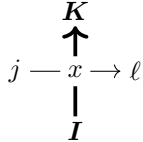

\centering
\gammaice{\bm{I}}{j}{\bm{K}}{\ell}
\caption{$\Gamma$ vertex with spectral parameter $x$}
\label{figure2.1}
\end{figure}

\begin{table}[!ht]
\centering
\begin{tabular}{|c|c|c|c|} 
 \hline
 Condition & $j=\ell, \ell=0$  &  $j=\ell,\ell\in [n]$  &  $j=\ell,\ell\in [\overline{n}]$\\
 \hline
 Boltzmann weight & $\frac{q^{-\boldsymbol{I}_{[1,n]}}-sx q^{\boldsymbol{I}_{[\overline{n},\overline{1}]}}}{1-sx}$ & $\frac{(sq^{\boldsymbol{I}_l}-x)q^{-\boldsymbol{I}_{\leq \ell}}}{s(1-sx)}$ & $\frac{s(s q^{\boldsymbol{I}_{\ell}}-x) q^{\boldsymbol{I}_{>\ell}}}{1-sx}$ \\
 \hline
 Condition & $j<\ell, \ell=0$ &  $j<\ell,\ell\in [n]$   & $j<\ell,\ell\in [\overline{n}]$   \\
 \hline
 Boltzmann weight & $\frac{xq^{-\boldsymbol{I}_{[1,n]}}-s^2xq^{\boldsymbol{I}_{[\overline{n},\overline{1}]}}}{s(1-sx)}$ & $\frac{-x(1-q^{\boldsymbol{I}_{\ell}})q^{-\boldsymbol{I}_{\leq \ell}}}{s(1-sx)}$ & $\frac{-sx(1-q^{\boldsymbol{I}_{\ell}})q^{\boldsymbol{I}_{>\ell}}}{1-sx}$ \\
 \hline
 Condition & $j>\ell, \ell=0$ & $j>\ell,\ell\in [n]$ & $j>\ell,\ell\in [\overline{n}]$\\
 \hline
 Boltzmann weight &$\frac{q^{-\boldsymbol{I}_{[1,n]}}-s^2 q^{\boldsymbol{I}_{[\overline{n},\overline{1}]}}}{1-sx}$ & $\frac{-(1-q^{\boldsymbol{I}_{\ell}})q^{-\boldsymbol{I}_{\leq \ell}}}{1-s x}$ & $\frac{-s^2(1-q^{\boldsymbol{I}_{\ell}}) q^{\boldsymbol{I}_{>\ell}}}{1-sx}$ \\
 \hline
\end{tabular}
\caption{Boltzmann weights for a $\Gamma$ vertex with spectral parameter $x$}
\label{table2.1}
\end{table}

\begin{figure}[!ht]
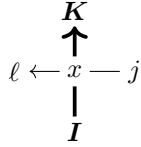

\centering
\deltaice{\bm{I}}{j}{\bm{K}}{\ell}
\caption{$\Delta$ vertex with spectral parameter $x$}
\label{figure2.2}
\end{figure}

\begin{table}[!ht]
\centering
\begin{tabular}{|c|c|c|c|} 
 \hline
 Condition & $j=\ell, \ell=0$ & $j=\ell,\ell\in [n]$ & $j=\ell,\ell\in [\overline{n}]$ \\
 \hline
 Boltzmann weight & $\frac{q^{-\boldsymbol{I}_{[\overline{n},\overline{1}]}}-sx q^{\boldsymbol{I}_{[1,n]}}}{1-sx}$ & $\frac{s(sq^{\boldsymbol{I}_{\ell}}-x)q^{\boldsymbol{I}_{< \ell}}}{1-sx}$ & $\frac{(1-s^{-1}xq^{-\boldsymbol{I}_{\ell}})q^{-\boldsymbol{I}_{>\ell}}}{1-sx}$ \\
 \hline
  Condition & $j<\ell, \ell=0$ & $j<\ell,\ell\in [n]$ & $j<\ell,\ell\in [\overline{n}]$\\
 \hline
 Boltzmann weight & $\frac{q^{-\boldsymbol{I}_{[\overline{n},\overline{1}]}}-s^2q^{\boldsymbol{I}_{[1,n]}}}{1-sx}$ & $\frac{-s^2(1-q^{\boldsymbol{I}_{\ell}})q^{\boldsymbol{I}_{<\ell}}}{1-sx}$ & $\frac{-(1-q^{\boldsymbol{I}_{\ell}})q^{-\boldsymbol{I}_{\geq\ell}}}{1-sx}$ \\
 \hline
 Condition & $j>\ell, \ell=0$ & $j>\ell,\ell\in [n]$ & $j>\ell,\ell\in [\overline{n}]$\\
 \hline
 Boltzmann weight & $\frac{x(s^{-1}q^{-\boldsymbol{I}_{[\overline{n},\overline{1}]}}-s q^{\boldsymbol{I}_{[1,n]}})}{1-sx}$ & $\frac{-sx(1-q^{\boldsymbol{I}_{\ell}})q^{\boldsymbol{I}_{<\ell}}}{1-s x}$ & $\frac{-s^{-1}x(1-q^{\boldsymbol{I}_{\ell}}) q^{-\boldsymbol{I}_{\geq \ell}}}{1-sx}$ \\
 \hline
\end{tabular}
\caption{Boltzmann weights for a $\Delta$ vertex with spectral parameter $x$}
\label{table2.2}
\end{table}

\begin{figure}[!ht]
\[
\begin{array}{|c|c|c|c|c|c|}
\hline
\text{Cap vertex} &\caps{+}{+} & \caps{i}{\overline{i}} & \caps{i}{i} & \caps{\overline{i}}{i} & \caps{\overline{i}}{\overline{i}} \\
\hline
\text{Boltzmann weight}  &  1 & t\phi(q^{-1\slash 2}x^{-1}) & 1-t\phi(q^{-1\slash 2}x^{-1}) & \phi(q^{-1\slash 2}x^{-1}) & 1-\phi(q^{-1\slash 2}x^{-1}) \\
\hline\end{array}\]
\caption{Boltzmann weights for a cap vertex with spectral parameter $x$, where $i \in[n]$}
\label{figure2.3}
\end{figure}

The model also depends on a signed permutation $\sigma\in B_n$ and a vector $\mu\in \{\pm 1,\cdots,\pm L\}^n$. The boundary conditions for the model are given as follows. On the left, for each $i\in [n]$, we assign color $0$ to the $(2i-1)$th row, and assign color $\sigma(i)$ to the $2i$th row. At the bottom, we assign the vector $\bm{e}_0$ to each vertical edge. At the top, for each $j\in \{1,\cdots,L\}$, we assign the vector $\sum_{i=1}^n (\mathbbm{1}_{\mu_i=j}\bm{e}_i+\mathbbm{1}_{\mu_i=-j}\bm{e}_{\overline{i}})$ to the vertical edge in the $j$th column. 

For example, when $n=2$, $L=4$, $\sigma=(1,2)$, $\mu=(4,-2)$, the model configuration is shown in Figure \ref{Configuration1}.

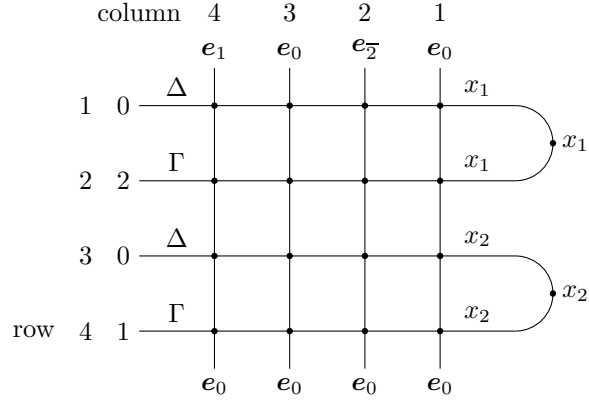
\begin{figure}[h]
\centering
\begin{tikzpicture}[scale=1]
\draw (0,1)--(5,1);
\draw (0,2)--(5,2);
\draw (0,3)--(5,3);
\draw (0,4)--(5,4);
\draw (1,0.5)--(1,4.5);
\draw (2,0.5)--(2,4.5);
\draw (3,0.5)--(3,4.5);
\draw (4,0.5)--(4,4.5);
\filldraw[black] (1,1) circle (1pt);
\filldraw[black] (2,1) circle (1pt);
\filldraw[black] (3,1) circle (1pt);
\filldraw[black] (4,1) circle (1pt);
\filldraw[black] (3,2) circle (1pt);
\filldraw[black] (2,2) circle (1pt);
\filldraw[black] (1,2) circle (1pt);
\filldraw[black] (4,2) circle (1pt);
\filldraw[black] (1,3) circle (1pt);
\filldraw[black] (2,3) circle (1pt);
\filldraw[black] (3,3) circle (1pt);
\filldraw[black] (4,3) circle (1pt);
\filldraw[black] (1,4) circle (1pt);
\filldraw[black] (2,4) circle (1pt);
\filldraw[black] (3,4) circle (1pt);
\filldraw[black] (4,4) circle (1pt);
\filldraw[black] (5.5,1.5) circle (1pt);
\filldraw[black] (5.5,3.5) circle (1pt);
\draw (5,1) arc(-90:90:0.5);
\draw (5,3) arc(-90:90:0.5);
\node at (0,1) [anchor=east] {$1$};
\node at (0,2) [anchor=east] {$0$};
\node at (0,3) [anchor=east] {$2$};
\node at (0,4) [anchor=east] {$0$};
\node at (-0.5,1) [anchor=east] {$4$};
\node at (-0.5,2) [anchor=east] {$3$};
\node at (-0.5,3) [anchor=east] {$2$};
\node at (-0.5,4) [anchor=east] {$1$};
\node at (-1,1) [anchor=east] {row};
\node at (0.5,1) [anchor=south] {$\Gamma$};
\node at (0.5,2) [anchor=south] {$\Delta$};
\node at (0.5,3) [anchor=south] {$\Gamma$};
\node at (0.5,4) [anchor=south] {$\Delta$};
\node at (4.5,1) [anchor=south] {$x_2$};
\node at (5.5,1.5) [anchor=west] {$x_2$};
\node at (4.5,2) [anchor=south] {$x_2$};
\node at (4.5,3) [anchor=south] {$x_1$};
\node at (4.5,4) [anchor=south] {$x_1$};
\node at (5.5,3.5) [anchor=west] {$x_1$};
\node at (1,4.5) [anchor=south] {$\bm{e}_1$};
\node at (2,4.5) [anchor=south] {$\bm{e}_0$};
\node at (3,4.5) [anchor=south] {$\bm{e}_{\overline{2}}$};
\node at (4,4.5) [anchor=south] {$\bm{e}_0$};
\node at (1,5) [anchor=south] {$4$};
\node at (2,5) [anchor=south] {$3$};
\node at (3,5) [anchor=south] {$2$};
\node at (4,5) [anchor=south] {$1$};
\node at (0,5) [anchor=south] {column};
\node at (1,0.5) [anchor=north] {$\bm{e}_0$};
\node at (2,0.5) [anchor=north] {$\bm{e}_0$};
\node at (3,0.5) [anchor=north] {$\bm{e}_0$};
\node at (4,0.5) [anchor=north] {$\bm{e}_0$};
\end{tikzpicture}
\caption{Model configuration when $n=2$, $L=4$, $\sigma=(1,2)$, $\mu=(4,-2)$}
\label{Configuration1}
\end{figure}

A state means an assignment of colors\slash vectors to all the edges in the rectangular lattice that is compatible with the boundary conditions as stated in the preceding. An admissible state means a state where the assignment of colors\slash vectors to the edges adjacent to each vertex in the lattice is one of the allowed assignments for that vertex (as listed in Table \ref{table2.1}, Table \ref{table2.2}, or Figure \ref{figure2.3} for $\Gamma$ vertex, $\Delta$ vertex, or cap vertex, respectively). The Boltzmann weight of an admissible state is defined as the product of the Boltzmann weights of all the vertices in the lattice corresponding to the state. The partition function of the model, denoted by $f_{\mu}^{\sigma}(\bm{x})$, is the sum of the Boltzmann weights for all admissible states.

For a $\Gamma$ vertex, we view the left and bottom edges adjacent to this vertex as input and the other two adjacent edges as output. For a $\Delta$ vertex, we view the right and bottom edges adjacent to this vertex as input and the other two adjacent edges as output. For a cap vertex, we view the bottom edge as input and the top edge as output. It can be checked that the Boltzmann weights for these vertices are stochastic, namely, the sum of the Boltzmann weights for a vertex with given input is equal to $1$.

\subsection{The $R$ vertices}\label{Sect.2.2}

In this subsection, we introduce three types of $R$ vertices called $\Gamma-\Gamma$ vertex, $\Delta-\Gamma$ vertex, and $\Delta-\Delta$ vertex. These $R$ vertices are rotated vertices, and depend on two spectral parameters. The Boltzmann weights for the three types of $R$ vertices with spectral parameters $x,y$ are given in Figures \ref{figure2.4}-\ref{figure2.6}.

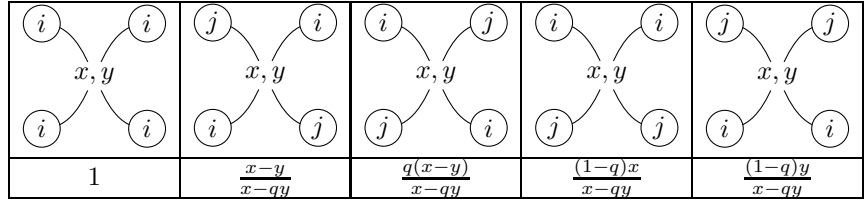
\begin{figure}[!ht]
\[\begin{array}{|c|c|c|c|c|}
\hline
\begin{tikzpicture}[scale=0.7]
\draw (0,0) to [out = 0, in = 180] (2,2);
\draw (0,2) to [out = 0, in = 180] (2,0);
\draw[fill=white] (0,0) circle (.35);
\draw[fill=white] (0,2) circle (.35);
\draw[fill=white] (2,0) circle (.35);
\draw[fill=white] (2,2) circle (.35);
\node at (0,0) {$i$};
\node at (0,2) {$i$};
\node at (2,2) {$i$};
\node at (2,0) {$i$};
\path[fill=white] (1,1) circle (.3);
\node at (1,1) {$x,y$};
\end{tikzpicture}&
\begin{tikzpicture}[scale=0.7]
\draw (0,0) to [out = 0, in = 180] (2,2);
\draw (0,2) to [out = 0, in = 180] (2,0);
\draw[fill=white] (0,0) circle (.35);
\draw[fill=white] (0,2) circle (.35);
\draw[fill=white] (2,0) circle (.35);
\draw[fill=white] (2,2) circle (.35);
\node at (0,0) {$i$};
\node at (0,2) {$j$};
\node at (2,2) {$i$};
\node at (2,0) {$j$};
\path[fill=white] (1,1) circle (.3);
\node at (1,1) {$x,y$};
\end{tikzpicture}&
\begin{tikzpicture}[scale=0.7]
\draw (0,0) to [out = 0, in = 180] (2,2);
\draw (0,2) to [out = 0, in = 180] (2,0);
\draw[fill=white] (0,0) circle (.35);
\draw[fill=white] (0,2) circle (.35);
\draw[fill=white] (2,0) circle (.35);
\draw[fill=white] (2,2) circle (.35);
\node at (0,0) {$j$};
\node at (0,2) {$i$};
\node at (2,2) {$j$};
\node at (2,0) {$i$};
\path[fill=white] (1,1) circle (.3);
\node at (1,1) {$x,y$};
\end{tikzpicture}&
\begin{tikzpicture}[scale=0.7]
\draw (0,0) to [out = 0, in = 180] (2,2);
\draw (0,2) to [out = 0, in = 180] (2,0);
\draw[fill=white] (0,0) circle (.35);
\draw[fill=white] (0,2) circle (.35);
\draw[fill=white] (2,0) circle (.35);
\draw[fill=white] (2,2) circle (.35);
\node at (0,0) {$j$};
\node at (0,2) {$i$};
\node at (2,2) {$i$};
\node at (2,0) {$j$};
\path[fill=white] (1,1) circle (.3);
\node at (1,1) {$x,y$};
\end{tikzpicture}&
\begin{tikzpicture}[scale=0.7]
\draw (0,0) to [out = 0, in = 180] (2,2);
\draw (0,2) to [out = 0, in = 180] (2,0);
\draw[fill=white] (0,0) circle (.35);
\draw[fill=white] (0,2) circle (.35);
\draw[fill=white] (2,0) circle (.35);
\draw[fill=white] (2,2) circle (.35);
\node at (0,0) {$i$};
\node at (0,2) {$j$};
\node at (2,2) {$j$};
\node at (2,0) {$i$};
\path[fill=white] (1,1) circle (.3);
\node at (1,1) {$x,y$};
\end{tikzpicture}\\
\hline
1
&\frac{x-y}{x-qy}
&\frac{q(x-y)}{x-qy}
&\frac{(1-q)x}{x-qy}
&\frac{(1-q)y}{x-qy}\\
\hline
\end{array}\]
\caption{Boltzmann weights for a $\Gamma-\Gamma$ vertex with spectral parameters $x,y$, where $i,j\in \{0\}\cup [\pm n]$ and $i<j$}
\label{figure2.4}
\end{figure}

\begin{figure}[!ht]
\[\begin{array}{|c|c|c|c|c|c|}
\hline
\begin{tikzpicture}[scale=0.7]
\draw (0,0) to [out = 0, in = 180] (2,2);
\draw (0,2) to [out = 0, in = 180] (2,0);
\draw[fill=white] (0,0) circle (.35);
\draw[fill=white] (0,2) circle (.35);
\draw[fill=white] (2,0) circle (.35);
\draw[fill=white] (2,2) circle (.35);
\node at (0,0) {$i$};
\node at (0,2) {$i$};
\node at (2,2) {$i$};
\node at (2,0) {$i$};
\path[fill=white] (1,1) circle (.3);
\node at (1,1) {$x,y$};
\end{tikzpicture}&
\begin{tikzpicture}[scale=0.7]
\draw (0,0) to [out = 0, in = 180] (2,2);
\draw (0,2) to [out = 0, in = 180] (2,0);
\draw[fill=white] (0,0) circle (.35);
\draw[fill=white] (0,2) circle (.35);
\draw[fill=white] (2,0) circle (.35);
\draw[fill=white] (2,2) circle (.35);
\node at (0,0) {$i$};
\node at (0,2) {$j$};
\node at (2,2) {$i$};
\node at (2,0) {$j$};
\path[fill=white] (1,1) circle (.3);
\node at (1,1) {$x,y$};
\end{tikzpicture}&
\begin{tikzpicture}[scale=0.7]
\draw (0,0) to [out = 0, in = 180] (2,2);
\draw (0,2) to [out = 0, in = 180] (2,0);
\draw[fill=white] (0,0) circle (.35);
\draw[fill=white] (0,2) circle (.35);
\draw[fill=white] (2,0) circle (.35);
\draw[fill=white] (2,2) circle (.35);
\node at (0,0) {$j$};
\node at (0,2) {$i$};
\node at (2,2) {$j$};
\node at (2,0) {$i$};
\path[fill=white] (1,1) circle (.3);
\node at (1,1) {$x,y$};
\end{tikzpicture}&
\begin{tikzpicture}[scale=0.7]
\draw (0,0) to [out = 0, in = 180] (2,2);
\draw (0,2) to [out = 0, in = 180] (2,0);
\draw[fill=white] (0,0) circle (.35);
\draw[fill=white] (0,2) circle (.35);
\draw[fill=white] (2,0) circle (.35);
\draw[fill=white] (2,2) circle (.35);
\node at (0,0) {$j$};
\node at (0,2) {$j$};
\node at (2,2) {$i$};
\node at (2,0) {$i$};
\path[fill=white] (1,1) circle (.3);
\node at (1,1) {$x,y$};
\end{tikzpicture}&
\begin{tikzpicture}[scale=0.7]
\draw (0,0) to [out = 0, in = 180] (2,2);
\draw (0,2) to [out = 0, in = 180] (2,0);
\draw[fill=white] (0,0) circle (.35);
\draw[fill=white] (0,2) circle (.35);
\draw[fill=white] (2,0) circle (.35);
\draw[fill=white] (2,2) circle (.35);
\node at (0,0) {$i$};
\node at (0,2) {$i$};
\node at (2,2) {$j$};
\node at (2,0) {$j$};
\path[fill=white] (1,1) circle (.3);
\node at (1,1) {$x,y$};
\end{tikzpicture}\\
\hline
1
&\frac{qxy-1}{xy-1}
&\frac{qxy-1}{q(xy-1)}
&\frac{1-q}{q(xy-1)}
&\frac{(1-q)xy}{xy-1}\\
\hline
\end{array}\]
\caption{Boltzmann weights for a $\Delta-\Gamma$ vertex with spectral parameters $x,y$: where $i,j\in \{0\}\cup [\pm n]$ and $i<j$}
\label{figure2.5}
\end{figure}

\begin{figure}[!ht]
\[\begin{array}{|c|c|c|c|c|c|}
\hline
\begin{tikzpicture}[scale=0.7]
\draw (0,0) to [out = 0, in = 180] (2,2);
\draw (0,2) to [out = 0, in = 180] (2,0);
\draw[fill=white] (0,0) circle (.35);
\draw[fill=white] (0,2) circle (.35);
\draw[fill=white] (2,0) circle (.35);
\draw[fill=white] (2,2) circle (.35);
\node at (0,0) {$i$};
\node at (0,2) {$i$};
\node at (2,2) {$i$};
\node at (2,0) {$i$};
\path[fill=white] (1,1) circle (.3);
\node at (1,1) {$x,y$};
\end{tikzpicture}&
\begin{tikzpicture}[scale=0.7]
\draw (0,0) to [out = 0, in = 180] (2,2);
\draw (0,2) to [out = 0, in = 180] (2,0);
\draw[fill=white] (0,0) circle (.35);
\draw[fill=white] (0,2) circle (.35);
\draw[fill=white] (2,0) circle (.35);
\draw[fill=white] (2,2) circle (.35);
\node at (0,0) {$i$};
\node at (0,2) {$j$};
\node at (2,2) {$i$};
\node at (2,0) {$j$};
\path[fill=white] (1,1) circle (.3);
\node at (1,1) {$x,y$};
\end{tikzpicture}&
\begin{tikzpicture}[scale=0.7]
\draw (0,0) to [out = 0, in = 180] (2,2);
\draw (0,2) to [out = 0, in = 180] (2,0);
\draw[fill=white] (0,0) circle (.35);
\draw[fill=white] (0,2) circle (.35);
\draw[fill=white] (2,0) circle (.35);
\draw[fill=white] (2,2) circle (.35);
\node at (0,0) {$j$};
\node at (0,2) {$i$};
\node at (2,2) {$j$};
\node at (2,0) {$i$};
\path[fill=white] (1,1) circle (.3);
\node at (1,1) {$x,y$};
\end{tikzpicture}&
\begin{tikzpicture}[scale=0.7]
\draw (0,0) to [out = 0, in = 180] (2,2);
\draw (0,2) to [out = 0, in = 180] (2,0);
\draw[fill=white] (0,0) circle (.35);
\draw[fill=white] (0,2) circle (.35);
\draw[fill=white] (2,0) circle (.35);
\draw[fill=white] (2,2) circle (.35);
\node at (0,0) {$j$};
\node at (0,2) {$i$};
\node at (2,2) {$i$};
\node at (2,0) {$j$};
\path[fill=white] (1,1) circle (.3);
\node at (1,1) {$x,y$};
\end{tikzpicture}&
\begin{tikzpicture}[scale=0.7]
\draw (0,0) to [out = 0, in = 180] (2,2);
\draw (0,2) to [out = 0, in = 180] (2,0);
\draw[fill=white] (0,0) circle (.35);
\draw[fill=white] (0,2) circle (.35);
\draw[fill=white] (2,0) circle (.35);
\draw[fill=white] (2,2) circle (.35);
\node at (0,0) {$i$};
\node at (0,2) {$j$};
\node at (2,2) {$j$};
\node at (2,0) {$i$};
\path[fill=white] (1,1) circle (.3);
\node at (1,1) {$x,y$};
\end{tikzpicture}\\
\hline
1
&\frac{y-x}{y-qx}
&\frac{q(y-x)}{y-qx}
&\frac{(1-q)x}{y-qx}
&\frac{(1-q)y}{y-qx}\\
\hline
\end{array}\]
\caption{Boltzmann weights for a $\Delta-\Delta$ vertex with spectral parameters $x,y$: where $i,j\in \{0\}\cup [\pm n]$ and $i<j$}
\label{figure2.6}
\end{figure}

For $(X,Y)\in\{(\Gamma,\Gamma),(\Delta,\Gamma),(\Delta,\Delta)\}$, we denote by $R_{XY}(\alpha,\beta,\gamma,\delta;x,y)$ the Boltzmann weight of an $X-Y$ vertex with spectral parameters $x,y$ and adjacent edges given in Figure \ref{RmatrixC}.
\begin{figure}[!h]
\centering
 \begin{tikzpicture}[scale=0.7]
\draw (0,0) to [out = 0, in = 180] (2,2);
\draw (0,2) to [out = 0, in = 180] (2,0);
\draw[fill=white] (0,0) circle (.35);
\draw[fill=white] (0,2) circle (.35);
\draw[fill=white] (2,0) circle (.35);
\draw[fill=white] (2,2) circle (.35);
\node at (0,0) {$\alpha$};
\node at (0,2) {$\beta$};
\node at (2,2) {$\gamma$};
\node at (2,0) {$\delta$};
\path[fill=white] (1,1) circle (.3);
\node at (1,1) {$x,y$};
\end{tikzpicture}
\caption{$R$ vertex}
\label{RmatrixC}
\end{figure}
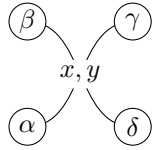

\subsection{The Yang-Baxter equations and the reflection equation}\label{Sect.2.3}

The following two propositions give three sets of Yang-Baxter equations and the reflection equation. These propositions are checked using a SAGE program and the fusion procedure as outlined in \cite[Appendix B]{BW}.

\begin{proposition}\label{YBE}
For any $(X,Y)\in\{(\Gamma,\Gamma),(\Delta,\Gamma),(\Delta,\Delta)\}$, the following holds. Assume that $S$ is $X$ vertex with spectral parameter $x$, $T$ is $Y$ vertex with spectral parameter $y$, and $R$ is $X-Y$ vertex with spectral parameters $x,y$. Then the partition functions of the following two configurations are equal for any fixed combination of colors $a,b,d,e\in \{0\}\cup [\pm n]$ and vectors $c,f\in \mathbb{N}^{2n}$.
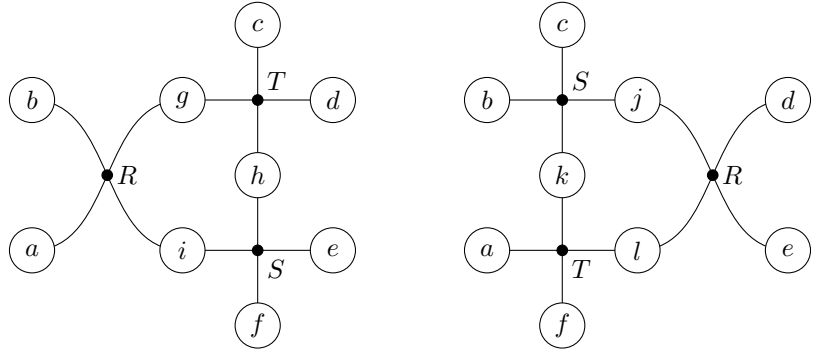
\begin{equation*}
\hfill
\begin{tikzpicture}[baseline=(current bounding box.center)]
  \draw (0,1) to [out = 0, in = 180] (2,3) to (4,3);
  \draw (0,3) to [out = 0, in = 180] (2,1) to (4,1);
  \draw (3,0) to (3,4);
  \draw[fill=white] (0,1) circle (.3);
  \draw[fill=white] (0,3) circle (.3);
  \draw[fill=white] (3,4) circle (.3);
  \draw[fill=white] (4,3) circle (.3);
  \draw[fill=white] (4,1) circle (.3);
  \draw[fill=white] (3,0) circle (.3);
  \draw[fill=white] (2,3) circle (.3);
  \draw[fill=white] (2,1) circle (.3);
  \draw[fill=white] (3,2) circle (.3);
  \node at (0,1) {$a$};
  \node at (0,3) {$b$};
  \node at (3,4) {$c$};
  \node at (4,3) {$d$};
  \node at (4,1) {$e$};
  \node at (3,0) {$f$};
  \node at (2,3) {$g$};
  \node at (3,2) {$h$};
  \node at (2,1) {$i$};
\filldraw[black] (3,3) circle (2pt);
\node at (3,3) [anchor=south west] {$T$};
\filldraw[black] (3,1) circle (2pt);
\node at (3,1) [anchor=north west] {$S$};
\filldraw[black] (1,2) circle (2pt);
\node at (1,2) [anchor=west] {$R$};
\end{tikzpicture}\qquad\qquad
\begin{tikzpicture}[baseline=(current bounding box.center)]
  \draw (0,1) to (2,1) to [out = 0, in = 180] (4,3);
  \draw (0,3) to (2,3) to [out = 0, in = 180] (4,1);
  \draw (1,0) to (1,4);
  \draw[fill=white] (0,1) circle (.3);
  \draw[fill=white] (0,3) circle (.3);
  \draw[fill=white] (1,4) circle (.3);
  \draw[fill=white] (4,3) circle (.3);
  \draw[fill=white] (4,1) circle (.3);
  \draw[fill=white] (1,0) circle (.3);
  \draw[fill=white] (2,3) circle (.3);
  \draw[fill=white] (1,2) circle (.3);
  \draw[fill=white] (2,1) circle (.3);
  \node at (0,1) {$a$};
  \node at (0,3) {$b$};
  \node at (1,4) {$c$};
  \node at (4,3) {$d$};
  \node at (4,1) {$e$};
  \node at (1,0) {$f$};
  \node at (2,3) {$j$};
  \node at (1,2) {$k$};
  \node at (2,1) {$l$};
\filldraw[black] (1,3) circle (2pt);
\node at (1,3) [anchor=south west] {$S$};
\filldraw[black] (1,1) circle (2pt);
\node at (1,1) [anchor=north west]{$T$};
\filldraw[black] (3,2) circle (2pt);
\node at (3,2) [anchor=west] {$R$};
\end{tikzpicture}
\end{equation*}
\end{proposition}

\begin{proposition}\label{Ref}
Assume that $S$ is $\Delta-\Delta$ vertex of spectral parameters $x,y$, $T$ is $\Delta-\Gamma$ vertex of spectral parameters $x,y$, $S'$ is $\Gamma-\Gamma$ vertex of spectral parameters $y,x$, $T'$ is $\Delta-\Gamma$ vertex of spectral parameters $y,x$, $C$ is cap vertex of spectral parameter $x$, and $C'$ is cap vertex of spectral parameter $y$. Then the partition functions of the following two configurations are equal for any fixed combination of colors $\epsilon_1,\epsilon_2,\epsilon_3,\epsilon_4\in \{0\}\cup[\pm n]$.
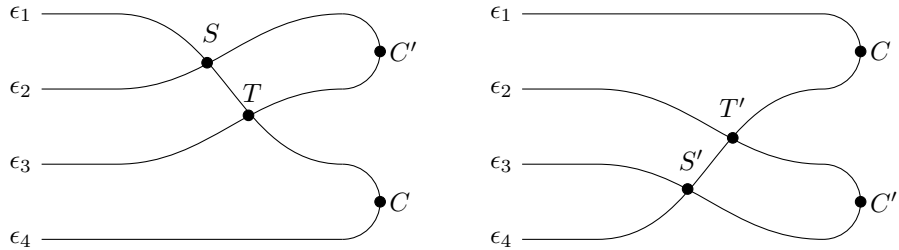
\begin{equation*}
\label{eqn:reflect1}
\hfill
\begin{tikzpicture}[baseline=(current bounding box.center)]
  \draw (0,0) to (4,0);
  \draw (0,1) to (1,1) to [out=0, in=180] (4,2);
  \draw (0,2) to (1,2) to [out=0, in=180] (4,3);
  \draw (0,3) to (1,3) to [out=0, in=180] (4,1);
  \node at (0,0) [anchor=east] {$\epsilon_4$};
  \node at (0,1) [anchor=east] {$\epsilon_3$};
  \node at (0,2) [anchor=east] {$\epsilon_2$};
  \node at (0,3) [anchor=east] {$\epsilon_1$};
  \draw (4,2) arc(-90:90:0.5);
  \draw (4,0) arc(-90:90:0.5);
  \filldraw[black] (4.5,0.5) circle (2pt);
  \filldraw[black] (4.5,2.5) circle (2pt);
  \filldraw[black] (2.2,2.35) circle (2pt);
  \filldraw[black] (2.75,1.65) circle (2pt);
  \node at (2.25,2.5) [anchor=south] {$S$};
  \node at (2.8, 1.7) [anchor=south] {$T$};
  \node at (4.5,0.5) [anchor=west] {$C$};
  \node at (4.5,2.5) [anchor=west] {$C'$};
\end{tikzpicture}
\quad\quad
\begin{tikzpicture}[baseline=(current bounding box.center)]
  \draw (0,3) to (4,3);
  \draw (0,1) to (1,1) to [out=0, in=180] (4,0);
  \draw (0,2) to (1,2) to [out=0, in=180] (4,1);
  \draw (0,0) to (1,0) to [out=0, in=180] (4,2);
  \node at (0,0) [anchor=east] {$\epsilon_4$};
  \node at (0,1) [anchor=east] {$\epsilon_3$};
  \node at (0,2) [anchor=east] {$\epsilon_2$};
  \node at (0,3) [anchor=east] {$\epsilon_1$};
  \draw (4,2) arc(-90:90:0.5);
  \draw (4,0) arc(-90:90:0.5);
  \filldraw[black] (4.5,0.5) circle (2pt);
  \filldraw[black] (4.5,2.5) circle (2pt);
  \filldraw[black] (2.2,0.67) circle (2pt);
  \filldraw[black] (2.8,1.35) circle (2pt);
  \node at (2.25,0.8) [anchor=south] {$S'$};
  \node at (2.8, 1.5) [anchor=south] {$T'$};
  \node at (4.5,0.5) [anchor=west] {$C'$};
  \node at (4.5,2.5) [anchor=west] {$C$};
\end{tikzpicture}
\end{equation*}
\end{proposition}

\section{Recursive relations for partition functions}\label{Sect.3}

In this section, we establish recursive relations for partition functions of the model introduced in Section \ref{Sect.2}. 

\subsection{Evaluation of the partition function for special boundary conditions}\label{Sect.3.1}

When $\sigma$ and $\mu$ satisfy the following two conditions, there is a unique admissible state:
\begin{itemize}
    \item[(a)] $\sigma(i)$ and $\mu_{|\sigma(i)|}$ have opposite signs for each $i\in [n]$;
    \item[(b)] $|\mu_{|\sigma(1)|}|\geq |\mu_{|\sigma(2)|}|\geq\cdots\geq |\mu_{|\sigma(n)|}|$.
\end{itemize}
The following theorem evaluates the partition function $f_{\mu}^{\sigma}(\bm{x})$ when $\sigma=Id$ (the identity permutation) and the above two conditions hold. 

\begin{theorem}\label{Thm3.1}
Assume that $\sigma=Id$, $|\mu_1|\geq |\mu_2|\geq \cdots\geq |\mu_n|$, and $\mu_i$ is negative for every $i\in [n]$. For each $j\in [L]$, we let $m_j(\mu):=|\{i\in [n]:\mu_i=-j\}|$. Then
\begin{equation*}
    f_{\mu}^{\sigma} (\bm{x}) = t^n \prod_{i=1}^n \phi(q^{-1\slash 2}x_i^{-1})\prod_{i=1}^n\Big(\frac{s-x_i}{s(1-sx_i)}\Big)^{L-\mu_i-1}\prod_{i=1}^n\Big(\frac{-sx_i}{1-sx_i}\Big)\prod_{j=1}^L(s^{-2};q^{-1})_{m_j(\mu)},
\end{equation*}
where 
\begin{equation*}
    (\alpha;q^{-1})_m:=(1-\alpha)(1-q^{-1}\alpha)\cdots(1-q^{-(m-1)}\alpha), \text{ for any } m\geq 1; \quad (\alpha;q^{-1})_0:=1.
\end{equation*}
\end{theorem}

\begin{proof}

Each admissible state can be interpreted as a collection of colored paths. Under the boundary conditions stated in the theorem, there is a unique admissible state: the path entering from the left of the second row must go through the $\Gamma$ vertices in the second row, take a U-turn at the cap vertex, move leftward in the first row until it reaches the $|\mu_{1}|$th column, and then move upward until reaching the top boundary; the path entering from the left of the fourth row must go through the $\Gamma$ vertices in the fourth row, take a U-turn at the cap vertex, move leftward in the third row until it reaches the $|\mu_{2}|$th column, and then move upward until reaching the top boundary; and so on. Evaluating the Boltzmann weight of this admissible state results in the conclusion of the theorem.

\end{proof}

\subsection{Recursive relations for partition functions}\label{Sect.3.2}

We have the following two recursive relations for partition functions.

\begin{theorem}\label{Thm3.2}
Assume that $1\leq i\leq n-1$ and $\sigma(i+1)>\sigma(i)$. Let $s_i\in B_n$ be the transposition $(i,i+1)$, and let $s_i\bm{x}=(x_1,\cdots,x_{i+1},x_i,\cdots,x_n)$. Then we have
\begin{equation*}
    q^{1+\mathbbm{1}_{\sigma(i+1)\in [n]}} f_{\mu}^{\sigma s_i}(\bm{x})=\frac{(q-1)x_{i+1}}{x_{i+1}-x_i}q^{\mathbbm{1}_{\sigma(i)\in [n]}} f_{\mu}^{\sigma}(\bm{x})+\frac{x_{i+1}-qx_i}{x_{i+1}-x_i} q^{\mathbbm{1}_{\sigma(i)\in [n]}} f_{\mu}^{\sigma}(s_i\bm{x}).
\end{equation*}
\end{theorem}

\begin{theorem}\label{Thm3.3}
Assume that $\sigma(n)\in [n]$. Let $s_n\in B_n$ be such that $s_n(n)=\overline{n}$ and $s_n(j)=j$ for every $j\in \{1,\cdots,n-1\}$, and let $s_n \bm{x}=(x_1,\cdots,x_{n-1},x_n^{-1})$. Then we have
\begin{eqnarray*}
  &&  q s^{-2L}\Big(\frac{1-sx_n}{s-x_n}\Big)^L f_{\mu}^{\sigma s_n}(\bm{x})\nonumber\\
  &=& \frac{(\sqrt{q}-\nu t x_n)(\sqrt{q}+\nu^{-1}x_n)}{t(1-x_n^2)}\Big(\frac{1-qx_n^2}{x_n^2-q}\Big(\frac{s-x_n}{1-sx_n}\Big)^L f_{\mu}^{\sigma}(s_n\bm{x})-\Big(\frac{1-sx_n}{s-x_n}\Big)^Lf_{\mu}^{\sigma}(\bm{x})\Big)\nonumber\\
  &&+\Big(\frac{1-sx_n}{s-x_n}\Big)^L f_{\mu}^{\sigma}(\bm{x}).
\end{eqnarray*}
\end{theorem}

Let $l(\cdot)$ be the length function on $B_n$. We define
\begin{equation*}
    F_{\mu}^{\sigma}(\bm{x}):=q^{l(\sigma)+\sum_{j=1}^n(n-j)\mathbbm{1}_{\sigma(j)\in [n]}} s^{-2L\sum_{j=1}^n\mathbbm{1}_{\sigma(j)\in [\overline{n}]}}\Big(\frac{t}{q}\Big)^{\sum_{j=1}^n \mathbbm{1}_{\sigma(j)\in [\overline{n}]}}\prod_{j=1}^n \Big(\frac{1-sx_j}{s-x_j}\Big)^L \prod_{j=1}^n\sqrt{\frac{x_j^2-q}{1-qx_j^2}} f_{\mu}^{\sigma}(\bm{x}).
\end{equation*}
From Theorems \ref{Thm3.2} and \ref{Thm3.3}, we can deduce that if $1\leq i\leq n-1$ and $\sigma(i+1)>\sigma(i)$,
\begin{equation*}
    F_{\mu}^{\sigma s_i}(\bm{x})=\frac{(q-1)x_{i+1}}{x_{i+1}-x_i} F_{\mu}^{\sigma}(\bm{x})+\frac{x_{i+1}-qx_i}{x_{i+1}-x_i}F_{\mu}^{\sigma}(s_i \bm{x});
\end{equation*}
if $\sigma(n)\in [n]$, 
\begin{equation*}
    F_{\mu}^{\sigma s_n}(\bm{x})=\frac{(1-\nu t q^{-1\slash 2}x_n)(1+\nu^{-1}q^{-1\slash 2} x_n)}{1-x_n^2}(F_{\mu}^{\sigma}(s_n\bm{x})-F_{\mu}^{\sigma}(\bm{x}))+\frac{t}{q}F_{\mu}^{\sigma}(\bm{x}).
\end{equation*}
For any rational function $g(\bm{x})$ and any $i\in [n]$, we define $s_i g(\bm{x}):=g(s_i\bm{x})$. Following the Noumi representation for affine Hecke algebra of type $\tilde{C}_n$ (see e.g. \cite{MR1715325,MR3369560}), for any rational function $g(\bm{x})$, we define
\begin{equation*}
    T_i g:=q g+\frac{x_{i+1}-qx_i}{x_{i+1}-x_i}(s_ig-g), \text{ for any } 1\leq i\leq n-1,
\end{equation*}
\begin{equation*}
    T_n g:=-ab g+\frac{(1-ax_n)(1-bx_n)}{1-x_n^2}(s_ng-g),
\end{equation*}
where $a,b\in\mathbb{C}$ are parameters. Thus taking $a=\nu t q^{-1\slash 2}$ and $b=-\nu^{-1}q^{-1\slash 2}$, we obtain that for any $i\in [n]$, $F_{\mu}^{\sigma s_i}=T_i F_{\mu}^{\sigma}$ as rational functions of $\bm{x}$. 

The rest of this subsection is devoted to the proofs of Theorems \ref{Thm3.2} and \ref{Thm3.3}.

\begin{proof}[Proof of Theorem \ref{Thm3.2}]

As shown in the following figure, we attach two $R$ vertices (a $\Delta-\Delta$ vertex $S$ with spectral parameters $x_i,x_{i+1}$ and a $\Delta-\Gamma$ vertex $T$ with spectral parameters $x_i,x_{i+1}$) to the left of the lattice configuration with spectral parameters $s_i\bm{x}$. Here we only show the $i$th and $(i+1)$th rows in the figure.
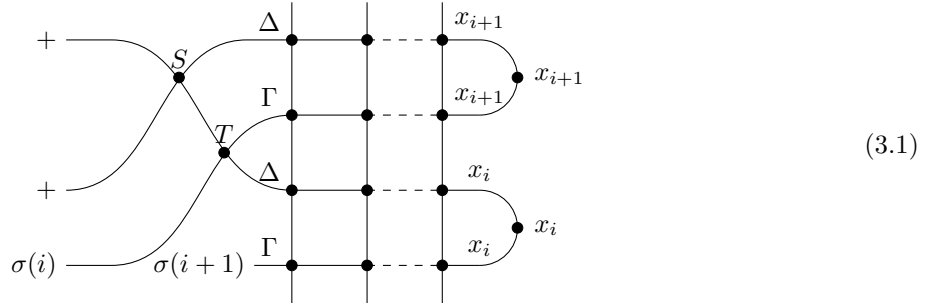
\begin{equation}\label{Eq3.1}
    \begin{tikzpicture}[baseline=(current bounding box.center)]
  \draw (0,0) to (0.6,0) to [out=0, in=180] (3,2) to (4,2);
  \draw (0,3) to (0.6,3) to [out=0, in=180] (3,1) to (4,1);
  \draw (0,1) to [out=0, in=180] (2.4,3) to (3,3) to (4,3);
  \draw (2.5,0) to (4,0);
  \draw [dashed] (4,2) to (5,2);
  \draw [dashed] (4,1) to (5,1);
  \draw [dashed] (4,3) to (5,3);
  \draw [dashed] (4,0) to (5,0);
  \draw (5,2) to (5.5,2);
  \draw (5,1) to (5.5,1);
  \draw (5,3) to (5.5,3);
  \draw (5,0) to (5.5,0);
  \draw (5.5,2) arc(-90:90:0.5);
  \draw (5.5,0) arc(-90:90:0.5);
  \draw (3,-0.5) to (3,3.5);
  \draw (4,-0.5) to (4,3.5);
  \draw (5,-0.5) to (5,3.5);
  \filldraw[black] (6,0.5) circle (2pt);
  \filldraw[black] (6,2.5) circle (2pt);
  \filldraw[black] (2.1,1.5) circle (2pt);
  \filldraw[black] (1.5,2.5) circle (2pt);
  \filldraw[black] (3,0) circle (2pt);
  \filldraw[black] (4,0) circle (2pt);
  \filldraw[black] (5,0) circle (2pt);
  \filldraw[black] (3,1) circle (2pt);
  \filldraw[black] (4,1) circle (2pt);
  \filldraw[black] (5,1) circle (2pt);
  \filldraw[black] (3,2) circle (2pt);
  \filldraw[black] (4,2) circle (2pt);
  \filldraw[black] (5,2) circle (2pt);
  \filldraw[black] (3,3) circle (2pt);
  \filldraw[black] (4,3) circle (2pt);
  \filldraw[black] (5,3) circle (2pt);
  \node at (2.1,1.5) [anchor=south] {$T$};
  \node at (1.5,2.5) [anchor=south] {$S$};
  \node at (0,0) [anchor=east] {$\sigma(i)$};
  \node at (0,1) [anchor=east] {$+$};
  \node at (2.5,0) [anchor=east] {$\sigma(i+1)$};
  \node at (0,3) [anchor=east] {$+$};
  \node at (5.5,0) [anchor=south] {$x_{i}$};
  \node at (5.5,1) [anchor=south] {$x_{i}$};
  \node at (5.5,2) [anchor=south] {$x_{i+1}$};
  \node at (5.5,3) [anchor=south] {$x_{i+1}$};
  \node at (2.7,0) [anchor=south] {$\Gamma$};
  \node at (2.7,1) [anchor=south] {$\Delta$};
  \node at (2.7,2) [anchor=south] {$\Gamma$};
  \node at (2.7,3) [anchor=south] {$\Delta$};
  \node at (6.1,0.5) [anchor=west] {$x_i$};
  \node at (6.1,2.5) [anchor=west] {$x_{i+1}$};
\end{tikzpicture}
\end{equation} 
Note that the only admissible configuration of the two $R$ vertices in (\ref{Eq3.1}) is given as follows:
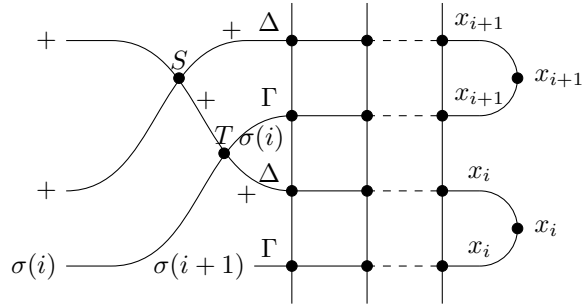
\begin{equation*}
    \begin{tikzpicture}[baseline=(current bounding box.center)]
  \draw (0,0) to (0.6,0) to [out=0, in=180] (3,2) to (4,2);
  \draw (0,3) to (0.6,3) to [out=0, in=180] (3,1) to (4,1);
  \draw (0,1) to [out=0, in=180] (2.4,3) to (3,3) to (4,3);
  \draw (2.5,0) to (4,0);
  \draw [dashed] (4,2) to (5,2);
  \draw [dashed] (4,1) to (5,1);
  \draw [dashed] (4,3) to (5,3);
  \draw [dashed] (4,0) to (5,0);
  \draw (5,2) to (5.5,2);
  \draw (5,1) to (5.5,1);
  \draw (5,3) to (5.5,3);
  \draw (5,0) to (5.5,0);
  \draw (5.5,2) arc(-90:90:0.5);
  \draw (5.5,0) arc(-90:90:0.5);
  \draw (3,-0.5) to (3,3.5);
  \draw (4,-0.5) to (4,3.5);
  \draw (5,-0.5) to (5,3.5);
  \filldraw[black] (6,0.5) circle (2pt);
  \filldraw[black] (6,2.5) circle (2pt);
  \filldraw[black] (2.1,1.5) circle (2pt);
  \filldraw[black] (1.5,2.5) circle (2pt);
  \filldraw[black] (3,0) circle (2pt);
  \filldraw[black] (4,0) circle (2pt);
  \filldraw[black] (5,0) circle (2pt);
  \filldraw[black] (3,1) circle (2pt);
  \filldraw[black] (4,1) circle (2pt);
  \filldraw[black] (5,1) circle (2pt);
  \filldraw[black] (3,2) circle (2pt);
  \filldraw[black] (4,2) circle (2pt);
  \filldraw[black] (5,2) circle (2pt);
  \filldraw[black] (3,3) circle (2pt);
  \filldraw[black] (4,3) circle (2pt);
  \filldraw[black] (5,3) circle (2pt);
  \node at (2.1,1.5) [anchor=south] {$T$};
  \node at (1.5,2.5) [anchor=south] {$S$};
  \node at (0,0) [anchor=east] {$\sigma(i)$};
  \node at (0,1) [anchor=east] {$+$};
  \node at (2.5,0) [anchor=east] {$\sigma(i+1)$};
  \node at (0,3) [anchor=east] {$+$};
  \node at (1.6,2.2) [anchor=west] {$+$};
  \node at (2.2,2.9) [anchor=south] {$+$};
  \node at (2.6,2) [anchor=north] {$\sigma(i)$};
  \node at (2.4,1.2) [anchor=north] {$+$};
  \node at (5.5,0) [anchor=south] {$x_{i}$};
  \node at (5.5,1) [anchor=south] {$x_{i}$};
  \node at (5.5,2) [anchor=south] {$x_{i+1}$};
  \node at (5.5,3) [anchor=south] {$x_{i+1}$};
  \node at (2.7,0) [anchor=south] {$\Gamma$};
  \node at (2.7,1) [anchor=south] {$\Delta$};
  \node at (2.7,2) [anchor=south] {$\Gamma$};
  \node at (2.7,3) [anchor=south] {$\Delta$};
  \node at (6.1,0.5) [anchor=west] {$x_i$};
  \node at (6.1,2.5) [anchor=west] {$x_{i+1}$};
\end{tikzpicture}
\end{equation*} 
Hence the partition function of the configuration in (\ref{Eq3.1}) is 
\begin{equation}\label{Eq3.8}
    R_{\Delta\Delta}(+,+,+,+;x_i,x_{i+1})R_{\Delta\Gamma}(\sigma(i),+,\sigma(i),+;x_i,x_{i+1})f_{\mu}^{\sigma}(s_i\bm{x}).
\end{equation}

By Proposition \ref{YBE}, we can push the two $R$ vertices to the right without changing the partition function. The new configuration is shown as follows:
\begin{equation}\label{Eq3.2}
    \begin{tikzpicture}[baseline=(current bounding box.center)]
  \draw (-0.5,2) to (1,2);
  \draw (-0.5,1) to (1,1);
  \draw (-0.5,3) to (1,3);
  \draw (-0.5,0) to (1,0);
  \draw [dashed] (1,2) to (2,2);
  \draw [dashed] (1,1) to (2,1);
  \draw [dashed] (1,3) to (2,3);
  \draw [dashed] (1,0) to (2,0);
  \draw (2,2) to (2.5,2);
  \draw (2,1) to (2.5,1);
  \draw (2,3) to (2.5,3);
  \draw (2,0) to (2.5,0);
  \draw (0,-0.5) to (0,3.5);
  \draw (1,-0.5) to (1,3.5);
  \draw (2,-0.5) to (2,3.5);
  \filldraw[black] (0,0) circle (2pt);
  \filldraw[black] (1,0) circle (2pt);
  \filldraw[black] (2,0) circle (2pt);
  \filldraw[black] (0,1) circle (2pt);
  \filldraw[black] (1,1) circle (2pt);
  \filldraw[black] (2,1) circle (2pt);
  \filldraw[black] (0,2) circle (2pt);
  \filldraw[black] (1,2) circle (2pt);
  \filldraw[black] (2,2) circle (2pt);
  \filldraw[black] (0,3) circle (2pt);
  \filldraw[black] (1,3) circle (2pt);
  \filldraw[black] (2,3) circle (2pt);
  \node at (-0.5,0) [anchor=east] {$\sigma(i+1)$};
  \node at (-0.5,2) [anchor=east] {$+$};
  \node at (-0.5,1) [anchor=east] {$\sigma(i)$};
  \node at (-0.5,3) [anchor=east] {$+$};
  \node at (2.8,0) [anchor=south] {$x_{i}$};
  \node at (2.8,1) [anchor=south] {$x_{i+1}$};
  \node at (2.8,2) [anchor=south] {$x_{i+1}$};
  \node at (2.8,3) [anchor=south] {$x_{i}$};
  \node at (2.3,0) [anchor=south] {$\Gamma$};
  \node at (2.3,1) [anchor=south] {$\Gamma$};
  \node at (2.3,2) [anchor=south] {$\Delta$};
  \node at (2.3,3) [anchor=south] {$\Delta$};
  \draw (2.5,0) to (6.5,0);
  \draw (2.5,1) to (3.5,1) to [out=0, in=180] (6.5,2);
  \draw (2.5,2) to (3.5,2) to [out=0, in=180] (6.5,3);
  \draw (2.5,3) to (3.5,3) to [out=0, in=180] (6.5,1);
  \draw (6.5,2) arc(-90:90:0.5);
  \draw (6.5,0) arc(-90:90:0.5);
  \filldraw[black] (7,0.5) circle (2pt);
  \filldraw[black] (7,2.5) circle (2pt);
  \filldraw[black] (4.7,2.35) circle (2pt);
  \filldraw[black] (5.25,1.65) circle (2pt);
  \node at (7.1,0.5) [anchor=west] {$x_i$};
  \node at (7.1,2.5) [anchor=west] {$x_{i+1}$};
  \node at (5.3,1.6) [anchor=south] {$T$};
  \node at (4.8,2.4) [anchor=south] {$S$};
\end{tikzpicture}
\end{equation}

By Proposition \ref{Ref}, the partition function of the configuration in (\ref{Eq3.2}) is equal to the partition function of the configuration in (\ref{Eq3.3}), which by Proposition \ref{YBE} is equal to the partition function of the configuration in (\ref{Eq3.4}) (where $S'$ is $\Gamma-\Gamma$ vertex with spectral parameters $x_{i+1},x_i$ and $T'$ is $\Delta-\Gamma$ vertex with spectral parameters $x_{i+1},x_i$).
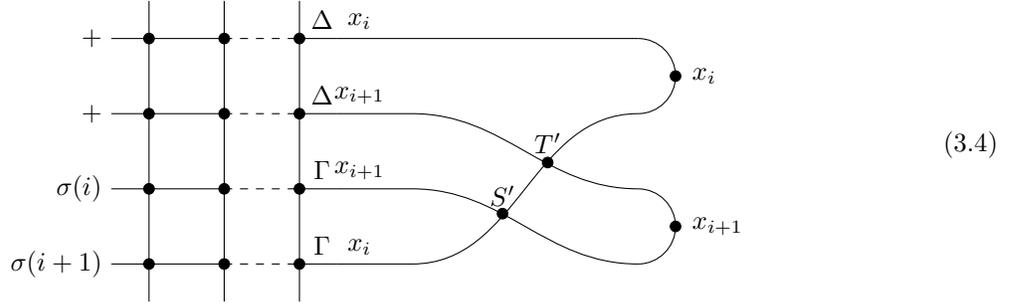
\begin{equation}\label{Eq3.3}
    \begin{tikzpicture}[baseline=(current bounding box.center)]
  \draw (-0.5,2) to (1,2);
  \draw (-0.5,1) to (1,1);
  \draw (-0.5,3) to (1,3);
  \draw (-0.5,0) to (1,0);
  \draw [dashed] (1,2) to (2,2);
  \draw [dashed] (1,1) to (2,1);
  \draw [dashed] (1,3) to (2,3);
  \draw [dashed] (1,0) to (2,0);
  \draw (2,2) to (2.5,2);
  \draw (2,1) to (2.5,1);
  \draw (2,3) to (2.5,3);
  \draw (2,0) to (2.5,0);
  \draw (0,-0.5) to (0,3.5);
  \draw (1,-0.5) to (1,3.5);
  \draw (2,-0.5) to (2,3.5);
  \filldraw[black] (0,0) circle (2pt);
  \filldraw[black] (1,0) circle (2pt);
  \filldraw[black] (2,0) circle (2pt);
  \filldraw[black] (0,1) circle (2pt);
  \filldraw[black] (1,1) circle (2pt);
  \filldraw[black] (2,1) circle (2pt);
  \filldraw[black] (0,2) circle (2pt);
  \filldraw[black] (1,2) circle (2pt);
  \filldraw[black] (2,2) circle (2pt);
  \filldraw[black] (0,3) circle (2pt);
  \filldraw[black] (1,3) circle (2pt);
  \filldraw[black] (2,3) circle (2pt);
  \node at (-0.5,0) [anchor=east] {$\sigma(i+1)$};
  \node at (-0.5,2) [anchor=east] {$+$};
  \node at (-0.5,1) [anchor=east] {$\sigma(i)$};
  \node at (-0.5,3) [anchor=east] {$+$};
  \node at (2.8,0) [anchor=south] {$x_{i}$};
  \node at (2.8,1) [anchor=south] {$x_{i+1}$};
  \node at (2.8,2) [anchor=south] {$x_{i+1}$};
  \node at (2.8,3) [anchor=south] {$x_{i}$};
  \node at (2.3,0) [anchor=south] {$\Gamma$};
  \node at (2.3,1) [anchor=south] {$\Gamma$};
  \node at (2.3,2) [anchor=south] {$\Delta$};
  \node at (2.3,3) [anchor=south] {$\Delta$};
 \draw (2.5,3) to (6.5,3);
  \draw (2.5,1) to (3.5,1) to [out=0, in=180] (6.5,0);
  \draw (2.5,2) to (3.5,2) to [out=0, in=180] (6.5,1);
  \draw (2.5,0) to (3.5,0) to [out=0, in=180] (6.5,2);
  \draw (6.5,2) arc(-90:90:0.5);
  \draw (6.5,0) arc(-90:90:0.5);
  \filldraw[black] (7,0.5) circle (2pt);
  \filldraw[black] (7,2.5) circle (2pt);
  \filldraw[black] (4.7,0.67) circle (2pt);
  \filldraw[black] (5.3,1.35) circle (2pt);
  \node at (7.1,0.5) [anchor=west] {$x_{i+1}$};
  \node at (7.1,2.5) [anchor=west] {$x_{i}$};
    \node at (5.3,1.35) [anchor=south] {$T'$};
  \node at (4.7,0.65) [anchor=south] {$S'$};
\end{tikzpicture}
\end{equation}
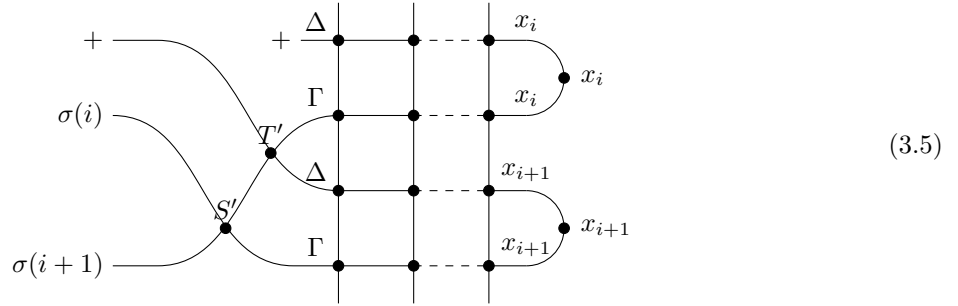
\begin{equation}\label{Eq3.4}
    \begin{tikzpicture}[baseline=(current bounding box.center)]
  \draw (0,0) to (0.6,0) to [out=0, in=180] (3,2) to (4,2);
  \draw (0,3) to (0.6,3) to [out=0, in=180] (3,1) to (4,1);
  \draw (2.5,3) to (4,3);
  \draw (0,2) to [out=0, in=180] (2.4,0) to (3,0) to (4,0);
  \draw [dashed] (4,2) to (5,2);
  \draw [dashed] (4,1) to (5,1);
  \draw [dashed] (4,3) to (5,3);
  \draw [dashed] (4,0) to (5,0);
  \draw (5,2) to (5.5,2);
  \draw (5,1) to (5.5,1);
  \draw (5,3) to (5.5,3);
  \draw (5,0) to (5.5,0);
  \draw (5.5,2) arc(-90:90:0.5);
  \draw (5.5,0) arc(-90:90:0.5);
  \draw (3,-0.5) to (3,3.5);
  \draw (4,-0.5) to (4,3.5);
  \draw (5,-0.5) to (5,3.5);
  \filldraw[black] (6,0.5) circle (2pt);
  \filldraw[black] (6,2.5) circle (2pt);
  \filldraw[black] (2.1,1.5) circle (2pt);
  \filldraw[black] (1.5,0.5) circle (2pt);
  \filldraw[black] (3,0) circle (2pt);
  \filldraw[black] (4,0) circle (2pt);
  \filldraw[black] (5,0) circle (2pt);
  \filldraw[black] (3,1) circle (2pt);
  \filldraw[black] (4,1) circle (2pt);
  \filldraw[black] (5,1) circle (2pt);
  \filldraw[black] (3,2) circle (2pt);
  \filldraw[black] (4,2) circle (2pt);
  \filldraw[black] (5,2) circle (2pt);
  \filldraw[black] (3,3) circle (2pt);
  \filldraw[black] (4,3) circle (2pt);
  \filldraw[black] (5,3) circle (2pt);
  \node at (2.1,1.5) [anchor=south] {$T'$};
  \node at (1.5,0.5) [anchor=south] {$S'$};
  \node at (0,0) [anchor=east] {$\sigma(i+1)$};
  \node at (2.5,3) [anchor=east] {$+$};
  \node at (0,2) [anchor=east] {$\sigma(i)$};
  \node at (0,3) [anchor=east] {$+$};
  \node at (5.5,0) [anchor=south] {$x_{i+1}$};
  \node at (5.5,1) [anchor=south] {$x_{i+1}$};
  \node at (5.5,2) [anchor=south] {$x_{i}$};
  \node at (5.5,3) [anchor=south] {$x_{i}$};
  \node at (2.7,0) [anchor=south] {$\Gamma$};
  \node at (2.7,1) [anchor=south] {$\Delta$};
  \node at (2.7,2) [anchor=south] {$\Gamma$};
  \node at (2.7,3) [anchor=south] {$\Delta$};
  \node at (6.1,0.5) [anchor=west] {$x_{i+1}$};
  \node at (6.1,2.5) [anchor=west] {$x_{i}$};
\end{tikzpicture}
\end{equation}
By considering possible configurations for $S',T'$, we obtain that the partition function of the configuration in (\ref{Eq3.4}) is
\begin{eqnarray}\label{Eq3.5}
  &&  R_{\Gamma\Gamma}(\sigma(i+1),\sigma(i),\sigma(i),\sigma(i+1);x_{i+1},x_i)R_{\Delta\Gamma}(\sigma(i),+,\sigma(i),+;x_{i+1},x_i) f_{\mu}^{\sigma}(\bm{x})\nonumber\\
  &&+R_{\Gamma\Gamma}(\sigma(i+1),\sigma(i),\sigma(i+1),\sigma(i);x_{i+1},x_i)R_{\Delta\Gamma}(\sigma(i+1),+,\sigma(i+1),+;x_{i+1},x_i)f_{\mu}^{\sigma s_i}(\bm{x}).
\end{eqnarray}

Comparing (\ref{Eq3.8}) and (\ref{Eq3.5}), we obtain that 
\begin{eqnarray*}
  &&  R_{\Delta\Delta}(+,+,+,+;x_i,x_{i+1})R_{\Delta\Gamma}(\sigma(i),+,\sigma(i),+;x_i,x_{i+1})f_{\mu}^{\sigma}(s_i\bm{x}) \nonumber\\
  &=& R_{\Gamma\Gamma}(\sigma(i+1),\sigma(i),\sigma(i),\sigma(i+1);x_{i+1},x_i)R_{\Delta\Gamma}(\sigma(i),+,\sigma(i),+;x_{i+1},x_i) f_{\mu}^{\sigma}(\bm{x})\nonumber\\
  &&+R_{\Gamma\Gamma}(\sigma(i+1),\sigma(i),\sigma(i+1),\sigma(i);x_{i+1},x_i)R_{\Delta\Gamma}(\sigma(i+1),+,\sigma(i+1),+;x_{i+1},x_i)f_{\mu}^{\sigma s_i}(\bm{x}),
\end{eqnarray*}
which by the Boltzmann weights in Section \ref{Sect.2.2} leads to the conclusion.

\end{proof}

\begin{proof}[Proof of Theorem \ref{Thm3.3}]

We denote $R:=\sigma(n)$, and consider the configuration that corresponds to $f_{\mu}^{\sigma}(\bm{x})$. Note that all the vertical edges at the bottom of the rectangular lattice carry the vector $\bm{e}_0$. Hence only the three states in Figure \ref{FishC} can appear for $\Gamma$ vertices in the $2n$th row. We also note that only the three states in Figure \ref{FishC2} can appear for the cap connecting the last two rows.

\begin{figure}[!h]
\[
\begin{array}{|c|c|c|c|c|c|}
\hline
\gammaicen{\bm{e}_0}{+}{\bm{e}_0}{+} &
\gammaicen{\bm{e}_0}{R}{\bm{e}_0}{R} &
\gammaicen{\bm{e}_0}{R}{\bm{e}_R}{+}\\
\hline
   1 & \frac{s-x_n}{s(1-sx_n)} & \frac{(1-s^2)x_n}{s(1-sx_n)} \\
\hline\end{array}\]
\caption{Boltzmann weights involved in the $2n$th row}
\label{FishC}
\end{figure}

\begin{figure}[!h]
\[
\begin{array}{|c|c|c|c|c|c|}
\hline
\text{Cap} &\capsN{+}{+} & \capsN{R}{\overline{R}} & \capsN{R}{R} \\
\hline
\text{Boltzmann weight}  &  1 & t\phi(q^{-1\slash 2}x_n^{-1}) & 1-t\phi(q^{-1\slash 2}x_n^{-1}) \\
\hline\end{array}\]
\caption{Boltzmann weights involved for the cap connecting the last two rows}
\label{FishC2}
\end{figure}

For each admissible state, we change the color in the $2n$th row from $+$ to $R$ and from $R$ to $+$. We also change the Boltzmann weights for the vertices in the $2n$th row into the Boltzmann weights of a $\Delta$ vertex with spectral parameter $x_n^{-1}$, and change the Boltzmann weights for the cap connecting the last two rows into the Boltzmann weights in Figure \ref{FishC4}. Note that only the three states listed in Figure \ref{FishCn} are involved in the $2n$th row for the new configuration. Let $Z_1$ be the partition function of the new configuration. We have
\begin{equation}\label{Eq3.15}
    f_{\mu}^{\sigma}(\bm{x})=\Big(\frac{s-x_n}{s(1-sx_n)}\Big)^LZ_1.
\end{equation}

\begin{figure}[!h]
\[
\begin{array}{|c|c|c|c|c|c|}
\hline
\text{New cap} &\newcaps{+}{R} & \newcaps{\overline{R}}{+} & \newcaps{R}{+} \\
\hline
\text{Boltzmann weight}  &  -1 & t\phi(q^{-1\slash 2}x_n^{-1}) &  1-t\phi(q^{-1\slash 2}x_n^{-1})  \\
\hline\end{array}\]
\caption{New Boltzmann weights for the cap connecting the last two rows}
\label{FishC4}
\end{figure}

\begin{figure}[!h]
\[
\begin{array}{|c|c|c|c|c|c|}
\hline
\gammaicenn{\bm{e}_0}{R}{\bm{e}_0}{R} &
\gammaicenn{\bm{e}_0}{+}{\bm{e}_0}{+} &
\gammaicenn{\bm{e}_0}{+}{\bm{e}_R}{R}\\
\hline
   \frac{s(1-sx_n)}{s-x_n} & 1 & -\frac{(1-s^2)x_n}{s-x_n} \\
\hline\end{array}\]
\caption{New Boltzmann weights involved in the $2n$th row}
\label{FishCn}
\end{figure}

Now we attach a $\Delta-\Delta$ vertex with spectral parameters $x_n^{-1},x_n$ to the left of the last two rows of the new configuration:
\begin{equation}\label{Eq3.6}
    \begin{tikzpicture}[baseline=(current bounding box.center)]
  \draw (0,0) to [out=0, in=-150] (1,0.5) to [out=30, in=180] (2,1);
  \draw (0,1) to [out=0, in=150] (1,0.5) to [out=-30, in=180] (2,0);
  \draw (2,0) to (3.5,0);
  \draw  [dashed] (3.5,0) to (4.5,0);
  \draw (4.5,0) to (5.5,0);
  \draw (2,1) to (3.5,1);
  \draw  [dashed] (3.5,1) to (4.5,1);
  \draw (4.5,1) to (5.5,1);
  \draw (5.5,0) to [out = 0, in = 180] (6,0.5);
  \draw (5.5,1) to [out = 0, in = 180] (6,0.5);
  \draw (2.5,-0.5) to (2.5,1.5);
  \draw (3.5,-0.5) to (3.5,1.5);
  \draw (4.5,-0.5) to (4.5,1.5);
  \filldraw[black] (2.5,0) circle (2pt);
  \filldraw[black] (3.5,0) circle (2pt);
  \filldraw[black] (4.5,0) circle (2pt);
  \filldraw[black] (2.5,1) circle (2pt);
  \filldraw[black] (3.5,1) circle (2pt);
  \filldraw[black] (4.5,1) circle (2pt);
  \filldraw[black] (6,0.5) circle (2pt);
  \node at (0,0) [anchor=east] {$+$};
  \node at (0,1) [anchor=east] {$+$};
  \filldraw[black] (1,0.5) circle (2pt);
  \node at (5,1) [anchor=south] {$\Delta$};
  \node at (5,0) [anchor=north] {$\Delta$};
  \node at (5.5,1) [anchor=south] {$x_n$};
  \node at (5.5,0) [anchor=north] {$x_n^{-1}$};
  \node at (0,0) [anchor=east] {$+$};
 \node at (6,0.5) [anchor=west] {$x_n$};
\end{tikzpicture}
\end{equation}
Note that the only admissible configuration for the $\Delta-\Delta$ vertex is given by
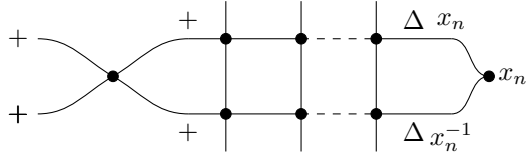
\begin{equation*}
    \begin{tikzpicture}[baseline=(current bounding box.center)]
  \draw (0,0) to [out=0, in=-150] (1,0.5) to [out=30, in=180] (2,1);
  \draw (0,1) to [out=0, in=150] (1,0.5) to [out=-30, in=180] (2,0);
  \draw (2,0) to (3.5,0);
  \draw  [dashed] (3.5,0) to (4.5,0);
  \draw (4.5,0) to (5.5,0);
  \draw (2,1) to (3.5,1);
  \draw  [dashed] (3.5,1) to (4.5,1);
  \draw (4.5,1) to (5.5,1);
  \draw (5.5,0) to [out = 0, in = 180] (6,0.5);
  \draw (5.5,1) to [out = 0, in = 180] (6,0.5);
  \draw (2.5,-0.5) to (2.5,1.5);
  \draw (3.5,-0.5) to (3.5,1.5);
  \draw (4.5,-0.5) to (4.5,1.5);
  \filldraw[black] (2.5,0) circle (2pt);
  \filldraw[black] (3.5,0) circle (2pt);
  \filldraw[black] (4.5,0) circle (2pt);
  \filldraw[black] (2.5,1) circle (2pt);
  \filldraw[black] (3.5,1) circle (2pt);
  \filldraw[black] (4.5,1) circle (2pt);
  \filldraw[black] (6,0.5) circle (2pt);
  \node at (0,0) [anchor=east] {$+$};
  \node at (0,1) [anchor=east] {$+$};
  \filldraw[black] (1,0.5) circle (2pt);
  \node at (5,1) [anchor=south] {$\Delta$};
  \node at (5,0) [anchor=north] {$\Delta$};
  \node at (5.5,1) [anchor=south] {$x_n$};
  \node at (5.5,0) [anchor=north] {$x_n^{-1}$};
  \node at (0,0) [anchor=east] {$+$};
  \node at (0,0) [anchor=east] {$+$};
  \node at (2,0) [anchor=north] {$+$};
  \node at (2,1) [anchor=south] {$+$};
  \node at (6,0.5) [anchor=west] {$x_n$};
\end{tikzpicture}
\end{equation*}
Hence the partition function of the configuration in (\ref{Eq3.6}) is $Z_1$. By Proposition \ref{YBE}, we can push the $\Delta-\Delta$ vertex to the right without changing the partition function. Hence $Z_1$ is equal to the partition function of the following configuration:
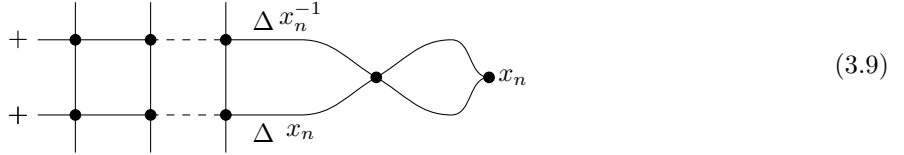
\begin{equation}
\begin{tikzpicture}[baseline=(current bounding box.center)]
\draw (3.5,0) to [out=0, in=-150] (4.5,0.5) to [out=30, in=180] (5.5,1);
 \draw (3.5,1) to [out=0, in=150] (4.5,0.5) to [out=-30, in=180] (5.5,0);
  \draw (0,0) to (1.5,0);
  \draw  [dashed] (1.5,0) to (2.5,0);
  \draw (2.5,0) to (3.5,0);
  \draw (0,1) to (1.5,1);
  \draw  [dashed] (1.5,1) to (2.5,1);
  \draw (2.5,1) to (3.5,1);
  \draw (5.5,0) to [out = 0, in = 180] (6,0.5);
  \draw (5.5,1) to [out = 0, in = 180] (6,0.5);
  \draw (0.5,-0.5) to (0.5,1.5);
  \draw (1.5,-0.5) to (1.5,1.5);
  \draw (2.5,-0.5) to (2.5,1.5);
  \filldraw[black] (0.5,0) circle (2pt);
  \filldraw[black] (1.5,0) circle (2pt);
  \filldraw[black] (2.5,0) circle (2pt);
  \filldraw[black] (0.5,1) circle (2pt);
  \filldraw[black] (1.5,1) circle (2pt);
  \filldraw[black] (2.5,1) circle (2pt);
  \filldraw[black] (6,0.5) circle (2pt);
  \node at (0,0) [anchor=east] {$+$};
  \node at (0,1) [anchor=east] {$+$};
  \filldraw[black] (4.5,0.5) circle (2pt);
  \node at (3,1) [anchor=south] {$\Delta$};
  \node at (3,0) [anchor=north] {$\Delta$};
  \node at (3.5,1) [anchor=south] {$x_n^{-1}$};
  \node at (3.5,0) [anchor=north] {$x_n$};
  \node at (0,0) [anchor=east] {$+$};
  \node at (6,0.5) [anchor=west] {$x_n$};
\end{tikzpicture}
\end{equation}

Now we introduce two auxiliary cap vertices $C_1$ and $C_2$. The Boltzmann weights for $C_1$ and $C_2$ are given in Figures  \ref{FishC5} and \ref{FishC6}, respectively.

\begin{figure}[!h]
\[
\begin{array}{|c|c|c|c|c|c|}
\hline
\text{New cap} &\capsC{+}{R}{C_1} & \capsC{\overline{R}}{+}{C_1} & \capsC{R}{+}{C_1} \\
\hline
\text{Boltzmann weight}  &  -1  & t\phi(q^{-1\slash 2}x_n) &  1-t\phi(q^{-1\slash 2}x_n) \\
\hline\end{array}\]
\caption{Boltzmann weights for the cap vertex $C_1$}
\label{FishC5}
\end{figure}

\begin{figure}[!h]
\[
\begin{array}{|c|c|c|c|c|c|}
\hline
\text{New cap} &\capsC{+}{\overline{R}}{C_2} & \capsC{\overline{R}}{+}{C_2} & \capsC{R}{+}{C_2} \\
\hline
\text{Boltzmann weight}  &  -1 & 1-\phi(q^{-1\slash 2}x_n) & \phi(q^{-1\slash 2}x_n) \\
\hline\end{array}\]
\caption{Boltzmann weights for the cap vertex $C_2$}
\label{FishC6}
\end{figure}

For any $\epsilon_1,\epsilon_2\in \{0\}\cup [\pm n]$, let $W(\epsilon_1,\epsilon_2), W_1(\epsilon_1,\epsilon_2), W_2(\epsilon_1,\epsilon_2)$ be the partition functions of the configurations in (\ref{Eq3.9})-(\ref{Eq3.11}), respectively (where the $R$ vertex in (\ref{Eq3.9}) is a $\Delta-\Delta$ vertex with spectral parameters $x_n^{-1},x_n$).
\begin{equation}\label{Eq3.9}
\begin{tikzpicture}[baseline=(current bounding box.center)]
  \draw (0,0) to [out=0, in=-150] (1,0.5) to [out=30, in=180] (2,1);
  \draw (0,1) to [out=0, in=150] (1,0.5) to [out=-30, in=180] (2,0);
  \draw (2,0) to [out = 0, in = 180] (2.5,0.5);
  \draw (2,1) to [out = 0, in = 180] (2.5,0.5);
  \filldraw[black] (2.5,0.5) circle (2pt);
  \node at (0,0) [anchor=east] {$\epsilon_1$};
  \node at (0,1) [anchor=east] {$\epsilon_2$};
  \node at (2.5,0.5) [anchor=west] {$x_n$};
  \filldraw[black] (1,0.5) circle (2pt);
\end{tikzpicture}
\end{equation}
\begin{equation}\label{Eq3.10}
\begin{tikzpicture}[baseline=(current bounding box.center)]
  \draw (0,0) to [out = 0, in = 180] (0.5,0.5);
  \draw[<-] (0,1) to [out = 0, in = 180] (0.5,0.5);
  \filldraw[black] (0.5,0.5) circle (2pt);
  \node at (0,0) [anchor=east] {$\epsilon_1$};
  \node at (0,1) [anchor=east] {$\epsilon_2$};
  \node at (0.5,0.5) [anchor=west] {$C_1$};
\end{tikzpicture}
\end{equation}
\begin{equation}\label{Eq3.11}
\begin{tikzpicture}[baseline=(current bounding box.center)]
  \draw (0,0) to [out = 0, in = 180] (0.5,0.5);
  \draw[<-] (0,1) to [out = 0, in = 180] (0.5,0.5);
  \filldraw[black] (0.5,0.5) circle (2pt);
  \node at (0,0) [anchor=east] {$\epsilon_1$};
  \node at (0,1) [anchor=east] {$\epsilon_2$};
   \node at (0.5,0.5) [anchor=west] {$C_2$};
\end{tikzpicture}
\end{equation}
It can be checked that for any $\epsilon_1,\epsilon_2\in \{0\}\cup [\pm n]$, 
\begin{eqnarray}\label{Eq3.12}
    W(\epsilon_1,\epsilon_2)&=&\Big(\frac{(1-q)x_n^2}{x_n^2-q}-\frac{x_n^2-1}{x_n^2-q}(1-t\phi(q^{-1\slash 2}x_n^{-1}))\Big)W_1(\epsilon_1,\epsilon_2) \nonumber\\
    && -\frac{q(x_n^2-1)t\phi(q^{-1\slash 2}x_n^{-1})}{x_n^2-q} W_2(\epsilon_1,\epsilon_2).
\end{eqnarray}

We denote by $Z_2,Z_2'$ the partition functions of the following two configurations (where we only show the last two rows of the lattice). 
\begin{equation}\label{Eq3.13}
\begin{tikzpicture}[baseline=(current bounding box.center)]
  \draw (0,0) to (1.5,0);
  \draw  [dashed] (1.5,0) to (2.5,0);
  \draw (2.5,0) to (3.5,0);
  \draw (0,1) to (1.5,1);
  \draw  [dashed] (1.5,1) to (2.5,1);
  \draw (2.5,1) to (3.5,1);
  \draw (3.5,0) to [out = 0, in = 180] (4,0.5);
  \draw (3.5,1) to [out = 0, in = 180] (4,0.5);
  \draw (0.5,-0.5) to (0.5,1.5);
  \draw (1.5,-0.5) to (1.5,1.5);
  \draw (2.5,-0.5) to (2.5,1.5);
  \filldraw[black] (0.5,0) circle (2pt);
  \filldraw[black] (1.5,0) circle (2pt);
  \filldraw[black] (2.5,0) circle (2pt);
  \filldraw[black] (0.5,1) circle (2pt);
  \filldraw[black] (1.5,1) circle (2pt);
  \filldraw[black] (2.5,1) circle (2pt);
  \filldraw[black] (4,0.5) circle (2pt);
  \node at (0,0) [anchor=east] {$+$};
  \node at (0,1) [anchor=east] {$+$};
  \node at (3,1) [anchor=south] {$\Delta$};
  \node at (3,0) [anchor=north] {$\Delta$};
  \node at (3.5,1) [anchor=south] {$x_n^{-1}$};
  \node at (3.5,0) [anchor=north] {$x_n$};
  \node at (0,0) [anchor=east] {$+$};
  \node at (4,0.5) [anchor=west] {$C_1$};
\end{tikzpicture}
\end{equation}
\begin{equation}\label{Eq3.14}
\begin{tikzpicture}[baseline=(current bounding box.center)]
  \draw (0,0) to (1.5,0);
  \draw  [dashed] (1.5,0) to (2.5,0);
  \draw (2.5,0) to (3.5,0);
  \draw (0,1) to (1.5,1);
  \draw  [dashed] (1.5,1) to (2.5,1);
  \draw (2.5,1) to (3.5,1);
  \draw (3.5,0) to [out = 0, in = 180] (4,0.5);
  \draw (3.5,1) to [out = 0, in = 180] (4,0.5);
  \draw (0.5,-0.5) to (0.5,1.5);
  \draw (1.5,-0.5) to (1.5,1.5);
  \draw (2.5,-0.5) to (2.5,1.5);
  \filldraw[black] (0.5,0) circle (2pt);
  \filldraw[black] (1.5,0) circle (2pt);
  \filldraw[black] (2.5,0) circle (2pt);
  \filldraw[black] (0.5,1) circle (2pt);
  \filldraw[black] (1.5,1) circle (2pt);
  \filldraw[black] (2.5,1) circle (2pt);
  \filldraw[black] (4,0.5) circle (2pt);
  \node at (0,0) [anchor=east] {$+$};
  \node at (0,1) [anchor=east] {$+$};
  \node at (3,1) [anchor=south] {$\Delta$};
  \node at (3,0) [anchor=north] {$\Delta$};
  \node at (3.5,1) [anchor=south] {$x_n^{-1}$};
  \node at (3.5,0) [anchor=north] {$x_n$};
  \node at (0,0) [anchor=east] {$+$};
  \node at (4,0.5) [anchor=west] {$C_2$};
\end{tikzpicture}
\end{equation}
By (\ref{Eq3.12}), we have
\begin{equation}\label{Eq3.16}
    Z_1=\Big(\frac{(1-q)x_n^2}{x_n^2-q}-\frac{x_n^2-1}{x_n^2-q}(1-t\phi(q^{-1\slash 2}x_n^{-1}))\Big)Z_2-\frac{q(x_n^2-1)t\phi(q^{-1\slash 2}x_n^{-1})}{x_n^2-q} Z_2'.
\end{equation}
Note that only the three states in Figure \ref{FishCm} can appear in the $2n$th row for the configuration in (\ref{Eq3.13}), and only the three states in Figure \ref{FishCmm} can appear in the $2n$th row for the configuration in (\ref{Eq3.14}).

\begin{figure}[!h]
\[
\begin{array}{|c|c|c|c|c|c|}
\hline
\gammaicenm{\bm{e}_0}{R}{\bm{e}_0}{R} &
\gammaicenm{\bm{e}_0}{+}{\bm{e}_0}{+} &
\gammaicenm{\bm{e}_0}{+}{\bm{e}_R}{R}\\
\hline
   \frac{s(s-x_n)}{1-s x_n} & 1 & \frac{1-s^2}{1-sx_n} \\
\hline\end{array}\]
\caption{Boltzmann weights involved in the $2n$th row for the configuration in (\ref{Eq3.13})}
\label{FishCm}
\end{figure}

\begin{figure}[!h]
\[
\begin{array}{|c|c|c|c|c|c|}
\hline
\gammaicenm{\bm{e}_0}{\overline{R}}{\bm{e}_0}{\overline{R}} &
\gammaicenm{\bm{e}_0}{+}{\bm{e}_0}{+} &
\gammaicenm{\bm{e}_0}{+}{\bm{e}_{\overline{R}}}{\overline{R}}\\
\hline
   \frac{s-x_n}{s(1-sx_n)} & 1 & \frac{(1-s^2)x_n}{s(1-sx_n)} \\
\hline\end{array}\]
\caption{Boltzmann weights involved in the $2n$th row for the configuration in (\ref{Eq3.14})}
\label{FishCmm}
\end{figure}

We compute $Z_2$ as follows. For the configuration in (\ref{Eq3.13}), we change the color in the $2n$th row from $+$ to $R$ and from $R$ to $+$ for each admissible state. We also change the Boltzmann weights for the vertices in the $2n$th row into the Boltzmann weights of a $\Gamma$ vertex with spectral parameter $x_n^{-1}$, and change the Boltzmann weights for $C_1$ into the Boltzmann weights of a cap vertex with spectral parameter $x_n^{-1}$. Note that for the changed configuration, only the three states listed in Figure \ref{Fishn1} are involved in the $2n$th row, and only the three states listed in Figure \ref{Fishn2} are involved in the cap connecting the last two rows. The partition function of the changed configuration is $f_{\mu}^{\sigma}(s_n\bm{x})$. Hence 
\begin{equation}\label{Eq3.17}
    Z_2=\Big(\frac{s(s-x_n)}{1-sx_n}\Big)^L f_{\mu}^{\sigma}(s_n\bm{x}).
\end{equation}

\begin{figure}[!h]
\[
\begin{array}{|c|c|c|c|c|c|}
\hline
\gammaicenn{\bm{e}_0}{+}{\bm{e}_0}{+} &
\gammaicenn{\bm{e}_0}{R}{\bm{e}_0}{R} &
\gammaicenn{\bm{e}_0}{R}{\bm{e}_R}{+}\\
\hline
   1 & \frac{1-sx_n}{s(s-x_n)} & -\frac{1-s^2}{s(s-x_n)} \\
\hline\end{array}\]
\caption{Boltzmann weights involved in the $2n$th row}
\label{Fishn1}
\end{figure}

\begin{figure}[!h]
\[
\begin{array}{|c|c|c|c|c|c|}
\hline
\text{Cap} &\capsm{+}{+} & \capsm{R}{\overline{R}} & \capsm{R}{R} \\
\hline
\text{Boltzmann weight}  &  1 & t\phi(q^{-1\slash 2}x_n) & 1-t\phi(q^{-1\slash 2}x_n) \\
\hline\end{array}\]
\caption{Boltzmann weights involved for the cap connecting the last two rows}
\label{Fishn2}
\end{figure}

We compute $Z_2'$ as follows. For the configuration in (\ref{Eq3.14}), we change the color in the $2n$th row from $+$ to $\overline{R}$ and from $\overline{R}$ to $+$ for each admissible state. We also change the Boltzmann weights for the vertices in the $2n$th row into the Boltzmann weights of a $\Gamma$ vertex with spectral parameter $x_n^{-1}$, and change the Boltzmann weights for $C_2$ into the Boltzmann weights of a cap vertex with spectral parameter $x_n^{-1}$. Note that for the changed configuration, only the three states listed in Figure \ref{Fishn3} are involved in the $2n$th row, and only the three states listed in Figure \ref{Fishn4} are involved in the cap connecting the last two rows. The partition function of the changed configuration is $f_{\mu}^{\sigma s_n}(s_n\bm{x})$. Hence 
\begin{equation}\label{Eq3.18}
    Z_2'=\Big(\frac{s-x_n}{s(1-sx_n)}\Big)^L f_{\mu}^{\sigma s_n}(s_n\bm{x}).
\end{equation}

\begin{figure}[!h]
\[
\begin{array}{|c|c|c|c|c|c|}
\hline
\gammaicenn{\bm{e}_0}{+}{\bm{e}_0}{+} &
\gammaicenn{\bm{e}_0}{\overline{R}}{\bm{e}_0}{\overline{R}} &
\gammaicenn{\bm{e}_0}{\overline{R}}{\bm{e}_{\overline{R}}}{+}\\
\hline
   1 & \frac{s(1-sx_n)}{s-x_n} & -\frac{(1-s^2)x_n}{s-x_n} \\
\hline\end{array}\]
\caption{Boltzmann weights involved in the $2n$th row}
\label{Fishn3}
\end{figure}

\begin{figure}[!h]
\[
\begin{array}{|c|c|c|c|c|c|}
\hline
\text{Cap} &\capsm{+}{+} & \capsm{\overline{R}}{R} & \capsm{\overline{R}}{\overline{R}} \\
\hline
\text{Boltzmann weight}  &  1 & \phi(q^{-1\slash 2}x_n) & 1-\phi(q^{-1\slash 2}x_n) \\
\hline\end{array}\]
\caption{Boltzmann weights involved for the cap connecting the last two rows}
\label{Fishn4}
\end{figure}

By (\ref{Eq3.15}) and (\ref{Eq3.16})-(\ref{Eq3.18}), we conclude that
\begin{eqnarray}
  &&  q s^{-2L}\Big(\frac{1-sx_n}{s-x_n}\Big)^L f_{\mu}^{\sigma s_n}(\bm{x})\nonumber\\
  &=& \frac{(\sqrt{q}-\nu t x_n)(\sqrt{q}+\nu^{-1}x_n)}{t(1-x_n^2)}\Big(\frac{1-qx_n^2}{x_n^2-q}\Big(\frac{s-x_n}{1-sx_n}\Big)^L f_{\mu}^{\sigma}(s_n\bm{x})-\Big(\frac{1-sx_n}{s-x_n}\Big)^Lf_{\mu}^{\sigma}(\bm{x})\Big)\nonumber\\
  &&+\Big(\frac{1-sx_n}{s-x_n}\Big)^L f_{\mu}^{\sigma}(\bm{x}).
\end{eqnarray}
    
\end{proof}

\newpage
\bibliographystyle{acm}
\bibliography{Vertexmodel.bib}

\end{document}